\documentclass{dcds-b}
\usepackage{amsmath}
\usepackage{paralist}
\usepackage{graphics} 
\usepackage{epsfig} 
\usepackage{subfigure}
\usepackage{graphicx}  \usepackage{epstopdf}
\usepackage[colorlinks=true]{hyperref}
\hypersetup{urlcolor=blue, citecolor=red}

\def\currentvolume{20}
 \def\currentissue{9}
  \def\currentyear{2015}
   \def\currentmonth{November}
    \def\ppages{X--XX}
     \def\DOI{10.3934/dcdsb.2015.20.xx}

\newtheorem{theorem}{Theorem}[section]
\newtheorem{corollary}{Corollary}

\newtheorem{lemma}{Lemma}[section]
\newtheorem{proposition}{Proposition}[section]

\theoremstyle{definition}
\newtheorem{definition}[theorem]{Definition}
\newtheorem{remark}{Remark}[section]

\def\ds{\displaystyle}
\def\e{{\varepsilon}}

\title[Shape Stability of OCPs for Systems of Hammerstein Type]{Shape Stability of Optimal Control Problems\\ in Coefficients for Coupled System\\ of Hammerstein Type}

\author[Olha P. Kupenko and Rosanna Manzo]{}


 \email{kogut\_olga@bk.ru}
 \email{rmanzo@unisa.it}

\begin{document}
\maketitle

\centerline{\scshape Olha P. Kupenko }
\medskip
{\footnotesize
\centerline{Dnipropetrovsk
Mining University}
 \centerline{Department of System Analysis and Control}
   \centerline{ Karl Marks av., 19, 49005 Dnipropetrovsk,
Ukraine}
\centerline{and}
   \centerline{Institute for
Applied System Analysis}
   \centerline{ National Academy of Sciences and Ministry
of Education and Science of Ukraine}
\centerline{ Peremogy av., 37/35, IPSA,
03056 Kyiv, Ukraine}
}

\medskip

\centerline{\scshape Rosanna Manzo}
\medskip
{\footnotesize
 \centerline{ Universit\`{a} degli Studi di Salerno}
   \centerline{Dipartimento di Ingegneria dell'Informazione, Ingegneria Elettrica e Matematica Applicata}
   \centerline{Via Giovanni Paolo II, 132, 84084 Fisciano (SA), Italy}
}

\bigskip
 \centerline{(Communicated by Gregoire Allaire)}

\begin{abstract}
In this paper we consider an optimal control
problem (OCP) for the coupled system of a nonlinear monotone Dirichlet problem with matrix-valued $L^\infty(\Omega;\mathbb{R}^{N\times N} )$-controls in coefficients
  and a nonlinear equation of Hammerstein type. Since problems
of this type have no solutions in general, we make a special
assumption on the coefficients of the state equation and introduce
the class of so-called solenoidal admissible controls. Using the direct
method in calculus of variations, we prove the existence of an optimal control. We also study the stability of the
optimal control problem with respect to the domain perturbation. In particular, we derive the sufficient conditions
of the Mosco-stability for the given class of OCPs.
\end{abstract}

\section{Introduction}
The aim of this paper is  to prove the
existence result for an optimal control problem (OCP) governed by the system of a nonlinear monotone elliptic equation
with homogeneous Dirichlet boundary conditions and a nonlinear equation of Hammerstein type, and  to provide sensitivity analysis of the considered optimization problem with respect to the domain perturbations.
As controls we consider the matrix of coefficients in the main part of the elliptic equation. We assume that admissible controls are measurable and uniformly bounded matrices of $L^\infty(\Omega;\mathbb{R}^{N\times N})$.

Systems with distributed parameters and optimal control problems for systems described by PDE, nonlinear integral and ordinary differential equations
have been widely studied by many authors (see for example \cite{Ivan_Mel,Lasiecka,Lions_0, Lurie,Zgurovski:99}). However, systems which contain equations of different types and optimization problems associated with them are still less well understood. In general case including as well  control and state constraints, such problems are rather complex and have no simple constructive solutions. The system, considered in the present paper, contains two equations: a nonlinear monotone elliptic equation with homogeneous Dirichlet boundary conditions and a nonlinear equation of Hammerstein type, which nonlinearly depends on the solution of the first object.  The optimal control problem we study here is to minimize the discrepancy between a given distribution $z_d\in L^p(\Omega)$ and a solution of Hammerstein equation $z=z(\mathcal{U},y)$, choosing an appropriate matrix of coefficients $\mathcal{U}\in U_{ad}$, i.e.
\begin{equation}
\label{0.1} I_{\Omega}(\mathcal{U},y,z)=\int_\Omega |z(x)-z_d(x)|^p\,dx \longrightarrow \inf
\end{equation}
subject to constrains
\begin{gather}
\label{0.2}
z + B F(y,z)=g\quad\mbox{ in }\Omega,\\
\label{0.3}
-\mathrm{div}\left(\mathcal{U}(x)[(\nabla y)^{p-2}]\nabla y
\right)+ |y|^{p-2}y=f\quad\mbox{ in }\Omega,\\
\mathcal{U}\in U_{ad},\quad y\in W^{1,p}_0(\Omega),
\label{0.4}
\end{gather}
where $U_{ad}\subset L^\infty(\Omega;\mathbb{R}^{N\times N})$ is a set of admissible controls, $B:L^q(\Omega)\to L^p(\Omega)$ is a positive linear operator, $F:W_0^{1,p}(\Omega)\times L^p(\Omega)\to L^q(\Omega)$ is an essentially nonlinear and non-monotone operator, $f\in W^{-1,q}(\Omega)$ and $g\in L^p(\Omega)$ are given distributions.

Since the range of optimal control problems in
coefficients is very wide, including as well optimal shape
design problems, optimization of certain evolution systems, some
problems originating in mechanics and others, this topic has been
widely studied by many authors. We mainly could mention Allaire
\cite{Allaire}, Buttazzo \& Dal Maso \cite{ButMas},
 Calvo-Jurado \& Casado-Diaz \cite{CalCas1},
 Haslinger \& Neittaanm¨aki
\cite{Haslinger}, Lions \cite{Lions_0},
Lurie \cite{Lurie}, Murat \cite{Murat1971}, Murat \& Tartar
\cite{MuratTartar1997}, Pironneau \cite{Pironneau}, Raytum
\cite{Raytum:89}, Sokolowski \&
Zolesio \cite{Sokolowski}, Tiba \cite{Tiba}, Mel'nik \& Zgurovsky
\cite{Zgurovski:99}.
In fact (see for instance \cite{Murat1971}), the most of
optimal control problems in coefficients for linear elliptic equations have no solution in general. It turns out that this
circumstance is the characteristic feature for the majority of optimal control problems in
coefficients. To overcome this difficulty, in present article, by analogy with \cite{CUO_09,OKogut2010,Kupenko2011}, we put some additional constrains on the set of admissible controls. Namely, we consider the matrix-valued controls from the so-called generalized solenoidal set. The elements of this set do not belong to any Sobolev space, but still are a little bit \textquotedblleft more regular \textquotedblright\, then those from $L^\infty$-class. Typically, the matrix of coefficients in the principle part of PDEs stands for anisotropic physical properties of media where the processes are studied. The main reason we introduce the class of generalized solenoidal controls is to achieve the desired well-posedness of the corresponding OCP and avoid the \textquotedblleft over regularity\textquotedblright\ of optimal characteristics. We give the precise definition of such controls in Section \ref{Sec 2} and prove that in this case the
original optimal control problem admits at least one solution.
It should be noticed that we do not involve the homogenization method and the
relaxation procedure in this process.

In practice, the equations of Hammerstein type appear as integral or integro-differential equations. The class of integral equations is very important for theory and applications, since there are less restrictions on smoothness of the desired solutions involved in comparison to those for the solutions of differential equations. Appearance of integral equations when solving boundary value problems is quite natural, since equations of such type bind together the values of known and unknown functions on bounded domains, in contrast to differential equations, where domains are infinitely small. It should be also mentioned here, that solution uniqueness is not typical for equations of Hammerstein type or optimization problems associated with such objects (see \cite{AMJA}). Indeed, this property requires rather strong assumptions on operators $B$ and $F$, which is rather restrictive in view of numerous applications (see \cite{VainLav}). The physical motivation of optimal control problems which are similar to those investigated in the present paper is widely discussed in \cite{AMJA, ZMN}.

As was pointed above, the principal feature of this problem is the
fact that an optimal solution for \eqref{0.1}--\eqref{0.4} does
not exist in general (see, e.g., \cite{ButMas}, \cite{CalCas1}, \cite{Murat1971},
 \cite{Raytum:89}). So here we have a typical
situation for the general optimal control theory. Namely, the
original control object is described by well-posed boundary value
problem, but the associated optimal control problem is ill-posed
and requires relaxation.

Since there is no good topology a priori given on the set of all
open subsets of $\mathbb{R}^N$, we study the stability properties
of the original control problem imposing some constraints on
domain perturbations. Namely, we consider two types of domain
perturbations: so-called topologically admissible perturbations
(following Dancer \cite{Dancer}), and perturbations in the
Hausdorff complementary topology (following Bucur and Zolesio
\cite{BuZo}). The asymptotical behavior of sets of admissible
triplets $\Xi_\e$ --- controls and the corresponding states --- under
domain perturbation is described in detail in Section
\ref{Sec 3}. In particular, we show that in this case the sequences of admissible triplets  to the perturbed problems are compact with respect to the
weak convergence in $L^\infty(D;\mathbb{R}^{N\times N})\times
W^{1,p}_0(D)\times L^p(D)$. Section \ref{Sec_4} is devoted to the stability
properties of optimal control problem \eqref{0.1}--\eqref{0.4}
under the domain perturbation. Our treatment of this question is
based on a new stability concept for optimal control problems (see for comparison \cite{CUO_09, CUO_12}). We
show that Mosco-stable optimal control problems
possess \textquotedblleft good\textquotedblright\, variational properties, which allow using optimal
solutions to the perturbed problems in \textquotedblleft simpler\textquotedblright\, domains as a
basis for the construction of suboptimal controls for the original
control problem. As a practical motivation of this approach we want to point out
that the \textquotedblleft real\textquotedblright\, domain $\Omega$ is never perfectly
smooth but contains microscopic asperities of  size
significantly smaller than characteristic length scale of the
domain. So a direct numerical computation of the solutions of
optimal control problems in such domains is extremely difficult. Usually it needs a very fine discretization mesh,
which means an enormous computation time, and such a computation
is often irrelevant. In view of the variational properties of
Mosco-stable problems  we can replace the
\textquotedblleft rough\textquotedblright\, domain $\Omega$ by a family of more \textquotedblleft regular\textquotedblright\, domains
$\left\{\Omega_\e\right\}_{\e>0}\subset D$ forming some admissible
perturbation and to approximate the original problem by the
corresponding perturbed problems \cite{Kogut}.

\section{Notation and preliminaries}
\label{Sec 1}
Throughout the paper $D$ and $\Omega$ are bounded
open subsets of $\mathbb{R}^N$, $N\ge 1$ and $\Omega\subset\subset D$. Let $\chi_\Omega$ be the characteristic function of the set $\Omega$  and let $\mathcal{L}^N(\Omega)$ be the $N$-dimensional Lebesgue measure of $\Omega$.   The space
$\mathcal{D}^\prime(\Omega)$ of distributions in $\Omega$ is the
dual of the space $C^\infty_0(\Omega)$. For real numbers $2\le
p<+\infty$,  and $1<
q<+\infty$ such that $1/p+1/q=1$, the space
$W^{1,p}_0(\Omega)$ is the closure of $C^\infty_0(\Omega)$ in
the Sobolev space $W^{1,p}(\Omega)$ with respect to the norm
\begin{equation}
\label{0}
\|y\|_{W_0^{1,p}(\Omega)}=\left(\int_\Omega \sum_{k=1}^{N}\left|\ds\frac{\partial y}{\partial x_i}\right|^p\,dx+\int_\Omega|y|^p\,dx\right)^{1/p},\;\forall\, y\in W_0^{1,p}(\Omega),
\end{equation}
 while $W^{-1,q}(\Omega)$ is the dual space of $W^{1,p}_0(\Omega)$.

For any vector field ${v}\in
L^q(\Omega;\mathbb{R}^N)$, the divergence
is an element of the space $W^{-1,\,q}(\Omega)$ defined by the
formula

\ \vspace*{-10pt}
\begin{equation}
\label{1.1}
\left<\mathrm{div}\,v,\varphi\right>_{W_0^{1,p}(\Omega)}
= -\int_\Omega
(v,\nabla\varphi)_{\mathbb{R}^N}\,dx,\quad \forall\,\varphi\in
W^{1,p}_0(\Omega),
\end{equation}
where $\left<\cdot,\cdot\right>_{
W^{1,p}_0(\Omega)}$ denotes the duality pairing between
$W^{-1,q}(\Omega)$ and $W^{1,p}_0(\Omega)$, and
$(\cdot,\cdot)_{\mathbb{R}^N}$ denotes the scalar product of two
vectors in $\mathbb{R}^N$.
A vector field $\mathbf{v}$ is said to be solenoidal, if
$\mathrm{div}\,\mathbf{v}=0$.

\textit{Monotone operators.} Let $\alpha$ and $\beta$ be constants such that $
0<\alpha\le\beta<+\infty$. We define
$M_{p}^{\alpha,\beta}(D)$ as the set of all square symmetric matrices
 $\mathcal{U}(x)=[a_{i\,j}(x)]_{1\le i,j\le N}$ in $L^\infty(D;\mathbb{R}^{N\times N})$ such
that the following conditions of growth, monotonicity, and strong coercivity are fulfilled:
\begin{gather}
\label{1.3} |a_{ij}(x)|\le\beta\quad\text{a.e. in }\ D,\ \forall\ i,j\in\{1,\dots,N\},\\
\label{1.4} \left(\mathcal{U}(x)([\zeta^{p-2}]\zeta-[\eta^{p-2}]\eta),\zeta-\eta\right)_{\mathbb{R}^N}\ge 0 \quad\text{a.e. in }\ D,\ \forall\,
\zeta,\eta\in \mathbb{R}^N,\\
\label{1.5} \left(\mathcal{U}(x)[\zeta^{p-2}]\zeta,\zeta\right)_{\mathbb{R}^N}=
 \sum\limits_{i,j=1}^N{a_{i\,j}(x)|\zeta_j|^{p-2}\,\zeta_j\,\zeta_i}\ge\alpha\,|\zeta|_p^p\quad\text{a.e in }\ D,
\end{gather}
where
$|\eta|_p=\left(\sum\limits_{k=1}^N |\eta_k|^p\right)^{1/p}$ is the H\"{o}lder norm of $\eta\in \mathbb{R}^N$ and
\begin{equation}
\label{1.5aa}
[\eta^{p-2}]=\mathrm{diag}\{|\eta_1|^{p-2},|\eta_2|^{p-2},\dots,|\eta_N|^{p
-2}\},\quad \forall \eta\in\mathbb{R}^N.
\end{equation}
\begin{remark}
\label{Rem 1.6} It is easy to see that $M_{p}^{\alpha,\beta}(D)$ is
a nonempty subset of $L^\infty(D;\mathbb{R}^{N\times N})$.
Indeed, as a representative of the set $M_{p}^{\alpha,\beta}(D)$ we can take any diagonal matrix of the form
$\mathcal{U}(x)=\mathrm{diag}\{\delta_1(x),\delta_2(x),\dots,\delta_N(x)\}$,
where functions $\delta_i(x)\in L^\infty(D)$ are such that $\alpha\le\delta_i(x)\le \beta$ a.e. in $D$
$\forall\,i\in\{1,\dots,N\}$ (see \cite{CUO_09}).
\end{remark}

Let us consider a nonlinear operator
$A:M_{p}^{\alpha,\beta}(D)\times W_0^{1,p}(\Omega)\to W^{-1,q}(\Omega)$ defined as
$$
A(\mathcal{U},y)=-\mathrm{div}\left(\mathcal{U}(x)[(\nabla y)^{p-2}]\nabla y\right)+|y|^{p-2}y,
$$
 or via the paring
\begin{gather*}
\langle A(\mathcal{U},y),v\rangle_{W_0^{1,p}(\Omega)}=\sum\limits_{i,j=1}^N\int_{\Omega}{\left(a_{ij}(x)\left|\displaystyle\frac{\partial
y}{\partial x_j}\right|^{p-2}\displaystyle\frac{\partial
y}{\partial x_j}\right) \displaystyle\frac{\partial v}{\partial
x_i}\,dx}\\+\int_{\Omega}{|y|^{p-2}y\,v\,dx},\quad \forall\,v\in W_0^{1,p}(\Omega).
\end{gather*}
In view of properties \eqref{1.3}--\eqref{1.5}, for every fixed matrix $\mathcal{U}\in M_{p}^{\alpha,\beta}(D)$, the operator $A(\mathcal{U},\cdot)$ turns out to be coercive, strongly monotone and demi-continuous in the following sense: $y_k\rightarrow y_0$ strongly in $W_0^{1,p}(\Omega)$ implies that $A(\mathcal{U},y_k)\rightharpoonup A(\mathcal{U},y_0) $
weakly in $W^{-1,q}(\Omega)$ (see \cite{Gaevski}). Then by
 well-known existence results for
nonlinear elliptic equations with strictly monotone
semi-continuous coercive operators (see \cite{Gaevski,Zgurovski:99}), the nonlinear Dirichlet boundary value problem
\begin{equation}
\label{1.7} A(\mathcal{U},y)=f\quad \text{ in }\quad \Omega,\qquad
y\in W^{1,p}_0(\Omega),
\end{equation}
admits a unique weak solution in $W^{1,p}_0(\Omega)$ for every fixed matrix $\mathcal{U}\in M_{p}^{\alpha,\beta}(D)$ and every distribution $f\in W^{-1,q}(D)$. Let us recall
that a function $y$ is the weak solution of \eqref{1.7} if
\begin{gather}
\label{1.8} y\in W^{1,p}_0(\Omega),\\
\label{1.9} \int_\Omega \left(\mathcal{U}(x)[(\nabla y)^{p-2}] \nabla y,
\nabla v\right)_{\mathbb{R}^N}\,dx + \int_\Omega |y|^{p-2}y
v\,dx=\int_\Omega f v\,dx,\; \forall\,v\in W^{1,p}_0(\Omega).
\end{gather}

\textit{System of nonlinear operator equations with an equation of Hammerstein type.}
Let $Y$ and $Z$ be Banach spaces, let $Y_0\subset Y$ be an arbitrary bounded set, and let $Z^\ast$ be the dual space to $Z$. To begin with we recall some useful properties of non-linear operators, concerning the solvability problem for  Hammerstein type equations and systems.
\begin{definition}\label{Def.1}
We say that the operator $G:D(G)\subset Z\to Z^\ast$ is radially continuous if for any $z_1,z_2\in X$ there exists $\e>0$ such that $z_1+\tau z_2\in D(G)$ for all $\tau\in [0,\e]$ and the real-valued function $[0,\e]\ni \tau\to\langle G(z_1+\tau z_2),z_2\rangle_Z$ is continuous.
\end{definition}
\begin{definition}
\label{Def.2}
An operator $G:Y\times Z\to Z^\ast$ is said to have a uniformly semi-bounded variation (u.s.b.v.) if for any bounded set  $Y_0\subset Y$
and any elements $z_1,z_2\in D(G)$ such that $\|z_i\|_Z\leq R$, $i=1,2$, the following inequality
\begin{equation}
\label{1.9.1}
\langle G(y,z_1)-G(y,z_2), z_1-z_2\rangle_{Z}\ge -\inf_{y\in Y_0}{C_{y}(R;\||z_1-z_2\||_Z)}
\end{equation}
holds true provided  the function $C_{y}:\mathbb{R}_+\times\mathbb{R}_+\to \mathbb{R}$ is continuous for each element $y\in Y_0$, and
$\ds\frac{1}{t}C_{y}(r,t)\to 0$ as $t\to 0$, $\forall\, r>0$. Here, $\||\cdot\||_Z$ is a seminorm on $Z$ such that $\||\cdot\||_Z$ is compact with respect to the norm $\|\cdot\|_Z$.
\end{definition}

It is worth to note that Definition \ref{Def.2} gives in fact a certain generalization of the classical monotonicity property. Indeed, if $C_{y}(\rho,r)\equiv 0$, then \eqref{1.9.1} implies the monotonicity property for the operator $G$ with respect to the second argument.
\begin{remark}\label{Rem 1.5}
Each operator $G:Y\times Z\to Z^\ast$ with u.s.b.v. possesses the following property (see for comparison Remark 1.1.2 in  \cite{AMJA}): if a set $K\subset Z$ is such that $\|z\|_{Z}\le k_1$ and $\langle G(y,z),z\rangle_{Z}\le k_2$ for all $z\in K$ and $y\in Y_0$, then there exists a constant $C>0$ such that $\|G(y,z)\|_{Z^\ast}\le C$, $\forall\,z\in K$ and $\forall y\in Y_0$.
\end{remark}

Let $B:Z^\ast\to Z$ and $F:Y\times Z\to Z^\ast$ be given operators such that the mapping $Z^\ast\ni z^\ast\mapsto B(z^\ast)\in Z$ is linear. Let $g\in Z$ be a given distribution. Then a typical operator equation of Hammerstein type can be represented as follows
\begin{equation}
\label{1.9.2}
z+B F(y,z)=g.
\end{equation}
The following existence result is well-known (see \cite[Theorem 1.2.1]{AMJA}).
\begin{theorem}\label{Th 1.1*}
Let $B:Z^\ast\to Z$ be a linear continuous positive operator such that it has the right inverse operator $B^{-1}_{r}:Z\to Z^\ast$. Let $F:Y\times Z\to Z^\ast$ be an operator with u.s.b.v such that
$F(y,\cdot):Z\to Z^\ast$ is radially continuous for each $y\in Y_0$ and the following inequality holds true
$$
\langle F(y,z)-B^{-1}_{r}g,z\rangle_{Z}\ge 0\quad\mbox{ if only } \|z\|_{Z}>\lambda>0,\;\lambda = const.
$$
Then the set
$$
\mathcal{H}(y)=\{z\in Z:\;z+BF(y,z)=g\ \text{in the sense of distributions }\}
$$
is non-empty and weakly compact for every fixed $y\in Y_0$ and $g\in Z$.
\end{theorem}
\vspace*{4pt}
\begin{definition}\label{def_new}
We say that

\begin{enumerate}
\item[($\mathfrak{M}$)] the operator $F:Y\times Z\to Z^\ast$ possesses the $\mathfrak{M}$-property if for any sequences $\{y_k\}_{k\in\mathbb{N}}\subset Y$ and $\{z_k\}_{k\in\mathbb{N}}\subset Z$ such that $y_k\to y$ strongly in $Y$ and $z_k\to z$ weakly in $Z$ as $k\to\infty$, the condition
\begin{equation}
\label{1*}
\lim_{k\to\infty}\langle F(y_k,z_k),z_k\rangle_{Z}=\langle F(y,z),z\rangle_{Z}
\end{equation}
implies that $z_k\to z$ strongly in $Z$.

\item[($\mathfrak{A}$)] the operator $F:Y\times Z\to Z^\ast$ possesses the $\mathfrak{A}$-property if
for any sequences $\{y_k\}_{k\in\mathbb{N}}\subset Y$ and $\{z_k\}_{k\in\mathbb{N}}\subset Z$ such that $y_k\to y$ strongly in $Y$ and $z_k\to z$ weakly in $Z$ as $k\to\infty$, the following relation \begin{equation}
\label{2*}
\liminf_{k\to\infty}\langle F(y_k,z_k),z_k\rangle_{Z}\ge
\langle F(y,z),z\rangle_{Z}
\end{equation}
holds true.
\end{enumerate}
\end{definition}

In what follows, we set $Y=W_0^{1,p}(\Omega)$, $Z=L^p(\Omega)$, and $Z^\ast=L^q(\Omega)$.

\subsection{Capacity} There are many ways to define the Sobolev
capacity. We use the notion of local
$p$-capacity which can be defined in the following way:
\begin{definition}
\label{Def 1.1} For a compact set $K$ contained in an arbitrary
ball $B$, capacity of $K$ in $B$, denoted by $C_p(K,B)$, is
defined as follows
\[
C_p(K,B)=\inf\left\{\int_B
|D\varphi|^p\,dx,\quad\forall\,\varphi\in C^\infty_0(B),\
\varphi\ge 1\ \text{ on }\ K\right\}.
\]
\end{definition}

For open sets contained in $B$ the capacity is defined by an
interior approximating procedure by compact sets (see
\cite{Heinonen}), and for arbitrary sets by an exterior
approximating procedure by open sets.

It is said that a property holds $p$-quasi everywhere (abbreviated
as $p$-q.e.) if it holds outside a set of $p$-capacity zero. It is
said that a property holds almost everywhere (abbreviated as
a.e.) if it holds outside a set of Lebesgue measure zero.

A function $y$ is called $p$-quasi--continuous if for any
$\delta>0$ there exists an open set $A_\delta$ such that
$C_p(A_\delta, B)<\delta$ and $y$ is continuous in $D\setminus
A_\delta$. We recall that any function $y\in W^{1,\,p}(D)$ has a
unique (up to a set of $p$-capacity zero) $p$-quasi continuous
representative. Let us recall the following results (see
\cite{Bagby,Heinonen}):
\begin{theorem}
\label{Th 1.1} Let $y\in W^{1,\,p}(\mathbb{R}^N)$. Then
$\left.y\right|_{\Omega}\in W^{1,\,p}_0(\Omega)$ provided
$y=0$ $p$-q.e. on $\Omega^c$ for a $p$-quasi-continuous
representative.
\end{theorem}
\begin{theorem}
\label{Th 1.2} Let $\Omega$ be a bounded open subset of
$\mathbb{R}^N$, and let $y\in W^{1,\,p}(\Omega)$. If $y=0$ a.e. in
$\Omega$, then $y=0$ $p$-q.e. in $\Omega$.
\end{theorem}

For these and other properties on quasi-continuous representatives,
the reader is referred to \cite{Bagby,Evans,Heinonen,Ziemer:89}.

\subsection{Convergence of sets}
In order to speak about \textquotedblleft domain perturbation\textquotedblright, we have
to prescribe a topology on the space of open subsets of $D$. To do this,
for the family of all open subsets of $D$, we define the Hausdorff
complementary topology, denoted by $H^c$, given by the metric:
\[
d_{H^c}(\Omega_1,\Omega_2)=\sup_{x \in
\mathbb{R}^N}\left|d(x,\Omega_1^c)-d(x,\Omega_2^c)\right|,
\]
where $\Omega_i^c$ are the complements of $\Omega_i$ in
$\mathbb{R}^N$.
\begin{definition}
\label{Def 1.3} We say that a sequence
$\left\{\Omega_\e\right\}_{\e>0}$ of open subsets of $D$ converges
to an open set $\Omega\subseteq D$ in $H^c$-topology, if
$d_{H^c}(\Omega_\e,\Omega)$ converges to $0$ as $\e\to 0$.
\end{definition}

The $H^c$-topology has some good properties, namely the space of
open subsets of $D$ is compact with respect to $H^c$-convergence,
and if $\Omega_\e\,\stackrel{H^c}{\longrightarrow}\,\Omega$, then
for any compact $K\subset\subset\Omega$ we have
$K\subset\subset\Omega_\e$ for $\e$ small enough. Moreover, a
sequence of open sets $\left\{\Omega_\e\right\}_{\e>0}\subset D$
$H^c$-converges to an open set $\Omega$, if and only if the
sequence of complements $\left\{\Omega_\e^c\right\}_{\e>0}$
converges to $\Omega^c$ in the sense of Kuratowski. We recall here
that a sequence $\left\{C_\e\right\}_{\e>0}$ of closed subsets of
$\mathbb{R}^N$ is said to be convergent to a closed set $C$ in the
sense of Kuratowski if the following two properties hold:
\begin{enumerate}
\item[$(K_1)$] for every $x\in C$, there exists a sequence
$\left\{x_\e\in C_\e\right\}_{\e>0}$ such that $x_\e\rightarrow x$
as $\e\to 0$;
\item[$(K_2)$] if $\left\{\e_k\right\}_{k\in \mathbb{N}}$ is a
sequence of indices converging to zero, $\left\{x_k\right\}_{k\in
\mathbb{N}}$ is a sequence such that $x_k\in C_{\e_k}$ for every
$k\in \mathbb{N}$, and $x_k$ converges to some $x\in
\mathbb{R}^N$, then $x\in C$.
\end{enumerate}
For these and other properties on $H^c$-topology, we refer to
\cite{Falconer}.

It is well known (see \cite{BuBu}) that in the case when $p> N$,
the $H^c$-convergence of open sets
$\left\{\Omega_\e\right\}_{\e>0}\subset D$ is equivalent to the
convergence in the sense of Mosco of the associated Sobolev
spaces.
\begin{definition}
 We say a sequence of spaces
$\left\{W^{1,\,p}_0(\Omega_\e)\right\}_{\e>0}$ converges in the
sense of Mosco to $W^{1,\,p}_0(\Omega)$ (see for comparison \cite{Mosco}) if the
following conditions are satisfied:
\begin{enumerate}
\item[$(M_1)$] for every $y\in W^{1,\,p}_0(\Omega)$ there exists a
sequence $\left\{y_\e\in W^{1,\,p}_0(\Omega_\e)\right\}_{\e>0}$
such that $\widetilde{y}_\e\rightarrow \widetilde{y}$ strongly in
$W^{1,\,p}(\mathbb{R}^N)$;

\item[$(M_2)$] if $\left\{\e_k\right\}_{k\in \mathbb{N}}$ is a
sequence converging to $0$ and $\left\{y_k\in W^{1,\,p}_0(\Omega_{\e_k})\right\}_{k\in \mathbb{N}}$
is a sequence such that  $\widetilde{y}_k\rightarrow \psi$
weakly in $W^{1,\,p}(\mathbb{R}^N)$, then there exists a function
$y\in W^{1,\,p}_0(\Omega)$ such that
$y=\left.\psi\right|_{\Omega}$.
\end{enumerate}
Hereinafter we denote by $\widetilde{y}_\e$ (respect.
$\widetilde{y}$) the zero-extension to $\mathbb{R}^N$ of a
function defined on $\Omega_\e$ (respect. on $\Omega$), that is,
$\widetilde{y}_\e=\widetilde{y}_\e \chi_{\Omega_\e}$ and
$\widetilde{y}=\widetilde{y} \chi_{\Omega}$.
\end{definition}

Following Bucur \& Trebeschi (see
\cite{Bucur_Treb}),  we have the following result.
\begin{theorem}
\label{Th 1.3} Let $\left\{\Omega_\e\right\}_{\e>0}$ be a sequence
of open subsets of $D$ such that
$\Omega_\e\,\stackrel{H^c}{\longrightarrow}\,\Omega$ and
$\Omega_\e\in \mathcal{W}_w(D)$ for every $\e>0$, with the class
$\mathcal{W}_w(D)$ defined as
\begin{multline}
\label{1.4*} \mathcal{W}_w(D)=\left\{\Omega\subseteq D\ :\
\forall\, x\in \partial\Omega, \forall\, 0<r<R<1; \right.\\\left.
\int_r^R \left(\frac{C_p(\Omega^c\cap
\overline{B(x,t)};B(x,2t))}{C_p(\overline{B(x,t)};B(x,2t))}\right)^{\frac{1}{p-1}}\frac{dt}{t}\ge
w(r,R,x)\right\},
\end{multline}
where $B(x,t)$ is the ball of radius $t$ centered at $x$,
and the function
\[
w:(0,1)\times(0,1)\times D\rightarrow \mathbb{R}^{+}
\]
is such that
\begin{enumerate}
\item[1.] $\lim_{r\to 0} w(r,R,x)=+\infty$, locally uniformly on $x\in
D$;

\item[2.] $w$ is a lower semicontinuous function in the third
argument.
\end{enumerate}

Then $\Omega\in \mathcal{W}_w(D)$ and the sequence of Sobolev
spaces $\left\{W^{1,\,p}_0(\Omega_\e)\right\}_{\e>0}$ converges in
the sense of Mosco to $W^{1,\,p}_0(\Omega)$.
\end{theorem}
\begin{theorem}
\label{Th 1.4} Let $N\ge p> N-1$ and let
$\left\{\Omega_\e\right\}_{\e>0}$ be a sequence of open subsets of
$D$ such that $\Omega_\e\,\stackrel{H^c}{\longrightarrow}\,\Omega$
and $\Omega_\e\in \mathcal{O}_l(D)$ for every $\e>0$, where the
class $\mathcal{O}_l(D)$ is defined as follows
\begin{equation}
\mathcal{O}_l(D)=\left\{\Omega\subseteq D\ :\ \sharp\Omega^c\le
l\right\}
\end{equation}
(here by $\sharp$ one denotes the number of connected components).
Then $\Omega\in\mathcal{O}_l(D)$ and the sequence of Sobolev
spaces $\left\{W^{1,\,p}_0(\Omega_\e)\right\}_{\e>0}$ converges in
the sense of Mosco to $W^{1,\,p}_0(\Omega)$.
\end{theorem}

In the meantime, the perturbation in $H^c$-topology (without some additional
assumptions) may be very irregular. It means that the continuity of
the mapping $\Omega\mapsto y_\Omega$, which associates to every
$\Omega$ the corresponding solution $y_{\,\Omega}$ of a Dirichlet
boundary problem \eqref{1.8}--\eqref{1.9}, may fail (see, for instance,
\cite{BuBu,DMaso_Ebob}). In view
of this, we introduce one more concept of the set convergence.
Following Dancer \cite{Dancer} (see also \cite{Daners}), we say
that
\begin{definition}
\label{Def 1.2} A sequence $\left\{\Omega_\e\right\}_{\e>0}$ of
open subsets of $D$ topologically converges to an open set
$\Omega\subseteq D$ ( in symbols
$\Omega_\e\,\stackrel{\mathrm{top}}{\longrightarrow}\, \Omega$) if
there exists a compact set $K_0\subset \Omega$ of $p$-capacity
zero $\left(C_p(K_0,D)=0\right)$ and a compact set $K_1\subset
\mathbb{R}^N$ of Lebesgue measure zero such that
\begin{enumerate}
\item[$(D_1)$] $\Omega^\prime\subset\subset \Omega\setminus K_0$
implies that $\Omega^\prime\subset\subset \Omega_\e$ for $\e$
small enough;
\item[$(D_2)$] for any open set $U$ with $\overline{\Omega}\cup
K_1\subset U$, we have $\Omega_\e\subset U$ for $\e$ small enough.
\end{enumerate}
\end{definition}

Note that without supplementary regularity assumptions on the
sets, there is no connection between topological set convergence, which is sometimes called \textquotedblleft convergence in the sense of compacts\textquotedblright\,
and the set convergence in the Hausdorff complementary topology (for examples and details see Remark \ref{Rem Ap.1} in the Appendix).


\section{Setting of the optimal control problem and existence result}
\label{Sec 2}

 Let $\xi_{\,1}$, $\xi_2$ be given functions of
$L^\infty(D)$ such that $0\le\xi_1(x)\le \xi_2(x)$ a.e. in $D$. Let
$\left\{Q_1,\dots,\,Q_N\right\}$ be a collection of nonempty
compact convex subsets of $W^{-1,\,q}(D)$. To define the class of
admissible controls, we introduce two sets
\begin{gather}
\label{2.4} U_{b}=\left\{\left. \mathcal{U}=[a_{i\,j}]\in
M_{p}^{\alpha,\beta}(D)\right| \xi_1(x)\leq a_{i\,j}(x)\leq
\xi_2(x)\;\mbox{a.e. in}\; D,\
\forall\,i,j=1,\dots,N\right\},\\[1ex]
\label{2.5} U_{sol}=\left\{\left. \mathcal{U}=[{{u}}_1,\dots,{{u}}_N]\in
M_{p}^{\alpha,\beta}(D)\right| \mathrm{div}\,{{u}}_i\in Q_i,\
\forall\,i=1,\dots,N\right\},
\end{gather}
assuming that the intersection $U_{b}\cap U_{sol}\subset
L^\infty(D;\mathbb{R}^{N\times N})$ is nonempty.
\begin{definition}
\label{Def 2.6} We say that a matrix $\mathcal{U}=[a_{i\,j}]$ is an
admissible control of  solenoidal type
if $\mathcal{U}\in U_{ad}:=U_{b}\cap U_{sol}$.
\end{definition}
\begin{remark}\label{rem 1.8}
As was shown in \cite{CUO_09} the set $U_{ad}$ is compact with respect to weak-$\ast$ topology of the space $L^\infty(D;\mathbb{R}^{N\times N})$.
\end{remark}

Let us consider the following optimal control problem:
\begin{gather}
\label{2.7} \text{Minimize }\ \Big\{I_\Omega(\mathcal{U},y,z)=\int_\Omega
|z(x)-z_d(x)|^p\,dx\Big\},
\end{gather}
subject to the constraints
\begin{gather}
\label{2.7a} \int_\Omega \left(\mathcal{U}(x)[(\nabla y)^{p-2}]\nabla y,
\nabla v\right)_{\mathbb{R}^N}\,dx + \int_\Omega |y|^{p-2}y
v\,dx=\left< f, v\right>_{W^{1,p}_0(\Omega)},\  \forall\,v\in W^{1,p}_0(\Omega),\\
\label{2.7b} \mathcal{U}\in U_{ad},\quad y\in W^{1,p}_0(\Omega),\\
\label{2.7c} \int_\Omega z \,\phi \, dx + \int_\Omega B F(y,z)\,\phi\,dx=\int_\Omega g\,\phi \,dx,
\end{gather}
where $f\in W^{-1,q}(D)$, $g\in L^p(D)$, and $z_d\in L^p(D)$ are given distributions.

Hereinafter, $\Xi_{sol}\subset L^\infty(D;\mathbb{R}^{N\times
N})\times W^{1,p}_0(\Omega)\times L^p(\Omega)$ denotes the set of all admissible triplets to the
optimal control problem \eqref{2.7}--\eqref{2.7c}.
\begin{definition}\label{def_tau}
Let $\tau$ be the topology on the set
$L^\infty(D;\mathbb{R}^{N\times N})\times W^{1,p}_0(\Omega)\times L^p(\Omega)$
which we define as a product of the weak-$\ast$ topology of
$L^\infty(D;\mathbb{R}^{N\times N})$, the weak topology of
$W^{1,p}_0(\Omega)$, and the weak topology of $L^p(\Omega)$.
\end{definition}

Further we use the following result (see \cite{CUO_09, KogutLeugering2011}).
\begin{proposition}
\label{Prop 1.16} For each $\mathcal{U}\in M_{p}^{\alpha,\beta}(D)$ and every $f\in
W^{-1,\,q}(D)$, a weak solution $y$ to variational
problem \eqref{2.7a}--\eqref{2.7b} satisfies the estimate
\begin{equation}
\label{1.17}
\|y\|^p_{W^{1,p}_0(\Omega)}\le
C\|f\|^q_{W^{-1,\,q}(D)},
\end{equation}
where  $C$ is a constant depending only on $p$ and $\alpha$.
\end{proposition}
\begin{proposition}\label{prop 1.15} Let
 $B:L^q(\Omega)\to L^p(\Omega)$ and $F:W_0^{1,p}(\Omega)\times L^p(\Omega)\to L^q(\Omega)$ be operators satisfying all conditions of Theorem \ref{Th 1.1*}.
Then the set
\begin{multline*}
\Xi_{sol}=\big\{(\mathcal{U},y,z)\in L^\infty(D;\mathbb{R}^{N\times N})\times W_0^{1,p}(\Omega)\times L^p(\Omega):\\ A(\mathcal{U},y)=f,\; z+B F(y,z)=g)\big\}
\end{multline*}
is nonempty for every $f\in W^{-1,q}(D)$ and $g\in L^p(D)$.
\end{proposition}
\begin{proof} The proof is given in Appendix.
\end{proof}

\begin{theorem}
\label{Th 2.8} Assume the following conditions hold:
\begin{itemize}
\item The operators $B:L^q(\Omega)\to L^p(\Omega)$ and $F:W_0^{1,p}(\Omega)\times L^p(\Omega)\to L^q(\Omega)$ satisfy conditions of Theorem \ref{Th 1.1*};
\item The operator $F(\cdot, z):W_0^{1,p}(\Omega)\to L^q(\Omega)$ is compact in the following sense: if $y_k\to y_0$ weakly in $W_0^{1,p}(\Omega)$, then  $F(y_k,z)\to F(y_0,z)$ strongly in $L^q(\Omega)$.
\end{itemize}
Then for every $f\in W^{-1,\,q}(D)$ and $g\in L^p(D)$,  the set $\Xi_{sol}$
is sequentially $\tau$-closed, i.e. if a sequence $\{(\mathcal{U}_k,y_k,z_k)\in \Xi_{sol}\}_{k\in\mathbb{N}}$ $\tau$-converges to a triplet $(\mathcal{U}_0,y_0,z_0)\in L^\infty(\Omega;\mathbb{R}^{N\times N})\times W^{1,p}_0(\Omega)\times L^p(\Omega)$, then $\mathcal{U}_0\in U_{ad}$, $y_0=y(\mathcal{U}_0)$, $z_0\in\mathcal{H}(y_0)$, and, therefore,
$(\mathcal{U}_0,y_0,z_0)\in\Xi_{sol}$.
\end{theorem}
\begin{proof}
Let $\{(\mathcal{U}_k,y_k,z_k)\}_{k\in\mathbb{N}}\subset \Xi_{sol}$ be any $\tau$-convergent sequence of admissible triplets to the optimal control problem \eqref{2.7}--\eqref{2.7c}, and let $(\mathcal{U}_0,y_0,z_0)$ be its $\tau$-limit in the sense of Definition \ref{def_tau}. Since the controls $\{\mathcal{U}_k\}_{k\in\mathbb{N}}$ belong to the set of solenoidal matrices $U_{sol}$ (see \eqref{2.5}), it follows from results given in \cite{OKogut2010,Kupenko2011} that $\mathcal{U}_0\in U_{ad}$ (see also Remark \ref{rem 1.8})  and $y_0=y(\mathcal{U}_0)$. It remains to show that  $z_0\in\mathcal{H}(y_0)$. To this end, we have to pass to the limit in equation
\begin{equation}\label{1.22}
z_k+BF(y_k,z_k)=g
\end{equation}
as $k\to\infty$ and get the limit pair $(y_0,z_0)$ is related by the equation
$
z_0+BF(y_0,z_0)=g.
$
With that in mind, let us rewrite equation  \eqref{1.22} in the following way
$$
B^\ast w_k +BF(y_k,B^\ast w_k)=g,
$$
where $w_k\in L^q(\Omega)$, $B^\ast:L^q(\Omega)\to L^p(\Omega)$ is the conjugate operator for $B$, i.e. $\langle B\nu, w\rangle_{L^q(\Omega)}=\langle B^\ast w,\nu\rangle_{L^q(\Omega)}$ and $B^\ast w_k=z_k$. Then, for every $k\in \mathbb{N}$, we have the equality
\begin{equation}
\label{1.18}
 \langle B^\ast w_k,w_k\rangle_{L^p(\Omega)}+\langle F(y_k,B^\ast w_k),B^\ast w_k\rangle_{L^p(\Omega)}=\langle g,w_k\rangle_{L^p(\Omega)}.
 \end{equation}
Taking into account the transformation
 $$
 \langle g,w_k\rangle_{L^p(\Omega)}=\langle B B^{-1}_r g,w_k\rangle_{L^p(\Omega)}=\langle B^{-1}_r  g,B^\ast w_k\rangle_{L^p(\Omega)},
 $$
we obtain
\begin{equation}
\label{1.18.1}
 \langle w_k,B^\ast w_k\rangle_{L^p(\Omega)}+\langle F(y_k,B^\ast w_k)-B^{-1}_r  g,B^\ast w_k\rangle_{L^p(\Omega)}=0.
\end{equation}
 The first term in \eqref{1.18.1} is strictly positive for every $w_k\neq 0$, hence, the second one must be negative.
In view of the initial assumptions, namely,
$$
\langle F(y,x)-B^{-1}_r  g,x\rangle_{L^p(\Omega)}\ge 0\ \text{ if only}\ \|x\|_{L^p(\Omega)}>\lambda,
$$
we conclude that
\begin{equation}
\label{7*}
\|B^\ast w_k\|_{L^p(\Omega)}=\|z_k\|_{L^p(\Omega)}\le \lambda.
\end{equation}
Since the linear positive operator $B^\ast$ cannot map unbounded sets into bounded ones, it follows that $\|w_k\|_{L^q(\Omega)}\le \lambda_1$.
As a result, see \eqref{1.18},  we have
\begin{equation}
\label{1.19}
\langle  F(y_k,B^\ast w_k),B^\ast w_k\rangle_{L^p(\Omega)}=-\langle B^\ast w_k,w_k\rangle_{L^p(\Omega)}+\langle g, w_k\rangle_{L^p(\Omega)},
\end{equation}
and, therefore, $\langle  F(y_k,B^\ast w_k),B^\ast w_k\rangle_{L^p(\Omega)}\le c_1$. Indeed, all terms in the right-hand side of \eqref{1.19} are bounded provided the sequence $\{w_k\}_{k\in\mathbb{N}}\subset L^q(\Omega)$ is bounded and operator $B$ is linear and continuous.  Hence, in view of Remark \ref{Rem 1.5}, we get
$$
\|F(y_k,B^\ast w_k)\|_{L^q(\Omega)}=\|F(y_k,z_k)\|_{L^q(\Omega)}\le c_2\
\text{ as }\ \|z_k\|_{L^p(\Omega)}\le \lambda.
$$
Since the right-hand side of \eqref{1.19}  does not depend on $y_k$, it follows that the constant $c_2>0$ does not depend on $y_k$ either.

Taking these arguments into account, we may suppose that up to a subsequence we have the weak convergence $F(y_k,z_k)\to \nu_0$ in $L^q(\Omega)$. As a result, passing to the limit in  \eqref{1.22}, by continuity of $B$, we finally get
\begin{equation}\label{1.20***}
z_0+B\nu_0=g.
\end{equation}
It remains to show that $\nu_0=F(y_0,z_0)$. Let us take an arbitrary element $z\in L^p(\Omega)$ such that $\|z\|_{L^p(\Omega)}\le \lambda$. Using the fact that $F$ is an operator with u.s.b.v., we have
$$
\langle F(y_k,z)-F(y_k,z_k),z-z_k\rangle_{L^p(\Omega)}\ge  -\inf_{y_k\in Y_0}C_{y_k}(\lambda;\||z-z_k\||_{L^p(\Omega)}),
$$
where $Y_0=\{y\in W_0^{1,p}(\Omega):\;y\mbox{ satisfies }\eqref{1.17}\}$, or, after transformation,
\begin{multline}
\label{1.21***}
\langle F(y_k,z),z-z_k\rangle_{L^p(\Omega)}-\langle F(y_k,z_k),z\rangle_{L^p(\Omega)}\\\ge \langle F(y_k,z_k),-z_k\rangle_{L^p(\Omega)}-\inf_{y_k\in Y_0}C_{y_k}(\lambda;\||z-z_k\||_{L^p(\Omega)}).
\end{multline}
Since $-z_k=BF(y_k,z_k)-g$, it follows from \eqref{1.21***} that
\begin{multline}
\label{1.21.0}
\langle F(y_k,z),z-z_k\rangle_{L^p(\Omega)}-\langle F(y_k,z_k),z\rangle_{L^p(\Omega)}  \\+\langle F(y_k,z_k),g\rangle_{L^p(\Omega)}\ge \langle F(y_k,z_k),B F(y_k,z_k) \rangle_{L^p(\Omega)}-\inf_{y_k\in Y_0}C_{y_k}(\lambda;\||z-z_k\||_{L^p(\Omega)}).
\end{multline}

In the meantime, due to the weak convergence $F(y_k,z_k)\to \nu_0$ in $L^q(\Omega)$ as $k\to\infty$, we arrive at the following obvious properties
\begin{gather}
\label{1.21.1}
\liminf_{k\to\infty}\langle F(y_k,z_k), B F(y_k,z_k) \rangle_{L^p(\Omega)}\ge \langle \nu_0, B\nu_0\rangle_{L^p(\Omega)},\\
\label{1.21.2a}
\lim_{k\to\infty}\langle F(y_k,z_k) ,z \rangle_{L^p(\Omega)}=
\langle \nu_0, z\rangle_{L^p(\Omega)},\\
\label{1.21.2}
\lim_{k\to\infty}\langle F(y_k,z_k),g \rangle_{L^p(\Omega)}=
\langle \nu_0, g\rangle_{L^p(\Omega)}.
\end{gather}
Moreover, the continuity of the function $C_{y_k}$ with respect to the second argument and the compactness property of operator $F$, which means that $F(y_k,z)\to F(y_0,z)$ strongly in $L^q(\Omega)$, lead to the conclusion
\begin{gather}
\label{1.21.3}
\lim_{k\to\infty} C_{y}(\lambda;\||z-z_k\||_{L^p(\Omega)})= C_{y}(\lambda;\||z-z_0\||_{L^p(\Omega)}),\quad\forall\,y\in Y_0,\\
\label{1.21.4}
\lim_{k\to\infty}\langle F(y_k,z),z-z_k\rangle_{L^p(\Omega)}= \langle F(y_0,z),z-z_0\rangle_{L^p(\Omega)}.
\end{gather}

As a result, using the properties \eqref{1.21.1}--\eqref{1.21.4}, we can pass to the limit in \eqref{1.21.0} as $k\to\infty$.  One gets
\begin{equation}
\label{1.21.5}
\langle F(y_0,z),z-z_0\rangle_{L^p(\Omega)}-\langle \nu_0,z+B\nu_0-g\rangle_{L^p(\Omega)}\ge -\inf_{y\in Y_0}C_{y}(\lambda;\||z-z_0\||_{L^p(\Omega)}).
\end{equation}
Since $B\nu_0-g=-z_0$ by \eqref{1.20***}, we can rewrite the inequality \eqref{1.21.5} as follows
\begin{align*}
\langle F(y_0,z)-\nu_0,z-z_0\rangle_{L^p(\Omega)}&\ge -\inf_{y\in Y_0}C_{y}(\lambda;\||z-z_0\||_{L^p(\Omega)})\\
&\ge - \inf_{y\in Y_0}C_{y}(\lambda;\||z-z_0\||_{L^p(\Omega)}).
\end{align*}
It remains to note that the operator $F$ is radially continuous for each $y\in Y_0$, and $F$ is the operator with u.s.b.v. (see Definitions \ref{Def.1} and \ref{Def.2}). Therefore, the last relation implies that $F(y_0,z_0)=\nu_0$ (see \cite[Theorem 1.1.2]{AMJA}) and, hence, equality \eqref{1.20***} finally takes the form
\begin{equation}\label{3*}
z_0+BF(y_0,z_0)=g.
\end{equation}
Thus, $z_0\in\mathcal{H}(y_0)$ and the triplet $(\mathcal{U}_0,y_0,z_0)$ is admissible for OCP \eqref{2.7}--\eqref{2.7c}. The proof is complete.
\end{proof}

\begin{remark}\label{Rem 1.7} In fact, as immediately follows from the proof of Theorem \ref{Th 2.8}, the set of admissible solutions $\Xi$ to the problem \eqref{2.7}--\eqref{2.7c} is sequentially $\tau$-compact.
\end{remark}

The next observation is important for our further analysis.
\begin{corollary}
\label{Rem 1.8}
Assume that all preconditions of Theorem \ref{Th 2.8} hold true. Assume also that the operator $F:W_0^{1,p}(\Omega)\times L^p(\Omega)\to L^q(\Omega)$ possesses $(\mathfrak{M})$ and $(\mathfrak{A})$ properties in the sense of Definition \ref{def_new}. Let $\left\{ y_k\right\}_{k\in \mathbb{N}}$  be a strongly convergent sequence in $W_0^{1,p}(\Omega)$. Then an arbitrary chosen sequence $\left\{ z_k\in \mathcal{H}(y_k)\right\}_{k\in \mathbb{N}}$ is relatively compact with respect to the strong topology of $L^p(\Omega)$, i.e. there exists an element $z_0\in \mathcal{H}(y_0)$ such that within a subsequence
$$
z_k\rightarrow z_0\ \text{ strongly in }\ L^p(\Omega)\ \text{ as }\ k\to\infty.
$$
\end{corollary}
\begin{proof}
The proof is given in Appendix. See also Remarks \ref{Rem 1.9} and \ref{Rem 1.10}.
\end{proof}

Now we are in a position to prove the existence result for the original optimal control problem \eqref{2.7}--\eqref{2.7c}.
\begin{theorem}
\label{Th 2.9} Assume that
$U_{ad}=U_{b}\cap U_{sol}\ne\emptyset$ and operators $B:L^q(\Omega)\to L^p(\Omega)$ and $F:W_0^{1,p}(\Omega)\times L^p(\Omega)\to L^q(\Omega)$ are as in Theorem~\ref{Th 2.8}. Then the optimal
control problem \eqref{2.7}--\eqref{2.7c} admits at least one
solution
\begin{gather*}
(\mathcal{U}^{opt}, y^{opt},z^{opt})\in \Xi_{sol}\subset L^\infty(\Omega;\mathbb{R}^{N\times
N})\times W^{1,p}_0(\Omega)\times L^p(\Omega),\\
I_\Omega(\mathcal{U}^{opt}, y^{opt},z^{opt})=\inf_{(\mathcal{U}, y,z)\in \Xi_{sol}}I_\Omega(\mathcal{U},y,z)
\end{gather*}
for each $f\in W^{-1,q}(D)$, $g\in L^p(D)$, and $z_d\in L^p(D)$.
\end{theorem}
\begin{proof}
Since the cost functional in \eqref{2.7} is bounded from below and, by Theorem~\ref{Th 1.1*}, the set of admissible solutions $\Xi_{sol}$ is nonempty, it follows that there exists a sequence $\{(\mathcal{U}_k,y_k,z_k)\}_{k\in \mathbb{N}}\subset \Xi_{sol}$ such that
$$
\lim_{k\to\infty}I_\Omega(\mathcal{U}_k,y_k,z_k)=\inf_{(\mathcal{U},y,z)\in \Xi_{sol}}I_\Omega(\mathcal{U},y,z).
$$
As it was mentioned in Remark \ref{Rem 1.7} the set of admissible solutions $\Xi$ to the problem \eqref{2.7}--\eqref{2.7c} is sequentially $\tau$-compact. Hence, there exists an admissible solution $(\mathcal{U}_0,y_0,z_0)$ such that, up to a subsequence, $(\mathcal{U}_k,y_k,z_k)\,\stackrel{\tau}{\rightarrow}\, (\mathcal{U}_0,y_0,z_0)$ as $k\to\infty$.
In order to show that $(\mathcal{U}_0,y_0,z_0)$  is an optimal solution of problem \eqref{2.7}--\eqref{2.7c}, it remains to make  use of the lower semicontinuity of the cost functional with respect to the $\tau$-convergence
\begin{align*}
I_\Omega(\mathcal{U}_0,y_0,z_0)&\le\liminf_{m\to\infty}I_\Omega(\mathcal{U}_{k_m},y_{k_m},z_{k_m})\\
&=\lim_{k\to\infty}I_\Omega(\mathcal{U}_k,y_k,z_k)=\inf_{(\mathcal{U}, y,z)\in \Xi_{sol}}I_\Omega(\mathcal{U},y,z).
\end{align*}
The proof is complete.
\end{proof}

\section{Domain perturbations for optimal control problem}
\label{Sec 3}
The aim of this section is to study the asymptotic behavior of solutions
$(\mathcal{U}_\e^{opt},y_\e^{opt},z_\e^{opt})$ to the optimal control problems
\begin{gather}
\label{3.1} I_{\,\Omega_\e}(\mathcal{U}_\e,y_\e,z_\e)=\int_{\Omega_\e}
|z_\e(x)-z_d(x)|^p\,dx
\longrightarrow\inf,\\
\label{3.2} -\mathrm{div}\,\left(\mathcal{U}_\e(x)[(\nabla y_\e)^{p-2}]
\nabla y_\e \right)+
|y_\e|^{p-2}y_\e=f \text{ in }
\Omega_\e,\\
\label{3.3} y_\e\in  W^{1,\,p}_0(\Omega_\e),\quad \mathcal{U}_\e\in U_{ad},\\
\label{3.4} z_\e + B F(y_\e,z_\e)=g\text{ in }
\Omega_\e,\; z_\e\in L^p(\Omega_\e)
\end{gather}
as $\e\to 0$ under some appropriate perturbations $\left\{\Omega_\e\right\}_{\e>0}$ of a fixed domain
$\Omega\subseteq D$. As before, we suppose that $f\in W^{-1,q}(D)$, $g\in L^p(D)$, and
$z_d\in L^p(D)$ are given functions. We assume that the set of admissible
controls $U_{ad}$ and, hence, the corresponding sets of admissible
solutions $\Xi_\e\subset L^\infty(D;\mathbb{R}^{N\times N})\times
W^{1,\,p}_0(\Omega_\e)\times L^p(\Omega_\e)$ are nonempty for every $\e>0$. We also assume that all conditions of Theorem \ref{Th 2.8} and Corollary \ref{Rem 1.8} hold true for every open subset $\Omega$ of $D$.

The following assumption is crucial for our further analysis.
\begin{enumerate}
\item [($\mathfrak{B}$)] The Hammerstein equation
  \begin{equation}
  \label{3.4.1}
  \int_D z \,\phi \, dx + \int_D B F(y,z)\,\phi\,dx=\int_D g\,\phi \,dx,
  \end{equation}
  possesses property $(\mathfrak{B})$, i.e. for any pair $(y,z)\in W_0^{1,p}(D)\times L^p(D)$ such that $z\in\mathcal{H}(y)$ and any sequence $\{y_k\}_{k\in\mathbb{N}}\subset W_0^{1,p}(D)$, strongly convergent in $W_0^{1,p}(D)$ to the element $y$, there exists a sequence $\{z_k\}_{k\in \mathbb{N}}\subset L^p(D)$ such that
\begin{gather*}
z_k\in \mathcal{H}(y_k),\quad\forall\, k\in \mathbb{N}\quad\mbox{ and }\quad
z_k\to z \mbox{ strongly in } L^p(D).
\end{gather*}
\end{enumerate}

\begin{remark}
As we have mentioned in Remark \ref{Rem 1.9}, under assumptions of Corollary \ref{Rem 1.8}, the set $\mathcal{H}(y)$ is non-empty and compact with respect to strong topology of $L^p(D)$ for every $y\in W_0^{1,p}(D)$. Hence, the $(\mathfrak{B})$-property obviously holds  true provided $\mathcal{H}(y)$ is a singleton (even if each of the sets $\mathcal{H}(y_k)$ contains more than one element). On the other hand, since we consider Hammerstein equation in rather general framework, it follows that without $(\mathfrak{B})$-property we cannot guarantee that every element of $\mathcal{H}(y)$ can be attained in strong topology by elements from $\mathcal{H}(y_k)$.
\end{remark}

Before we give the precise definition of the shape stability for
the above problem and admissible perturbations for open set
$\Omega$, we remark that neither the set convergence
$\Omega_\e\,\stackrel{H^c}{\longrightarrow}\,\Omega$ in the
Hausdorff complementary topology nor the topological set
convergence $\Omega_\e\,\stackrel{\mathrm{top}}{\longrightarrow}\,
\Omega$ is a sufficient condition to prove the shape stability of
the control problem \eqref{2.7}--\eqref{2.7b}. In general,  a
limit triplet for the sequence
$\left\{(\mathcal{U}_\e^{opt},y_\e^{opt},z_\e^{opt})\right\}_{\e>0}$, under
$H^c$-perturbations of $\Omega$, can be non-admissible to the
original problem \eqref{2.7}--\eqref{2.7c}. We refer to
\cite{DMaso_Mur} for simple counterexamples. So, we have to impose
some additional constraints on the moving domain. In view of this, we begin
with the following concepts:
\begin{definition}
\label{Def 3.1} Let $\Omega$ and $\left\{\Omega_\e\right\}_{\e>0}$
be open subsets of $D$. We say that the sets
$\left\{\Omega_\e\right\}_{\e>0}$ form  an $H^c$-admissible
perturbation of $\Omega$, if:
\begin{enumerate}
\item[(i)] $\Omega_\e\,\stackrel{H^c}{\longrightarrow}\,\Omega$ as
$\e\to 0$;
\item[(ii)] $\Omega_\e\in  \mathcal{W}_w(D)$ for every $\e>0$, where the class
$\mathcal{W}_w(D)$ is defined in \eqref{1.4*}.
\end{enumerate}
\end{definition}

\begin{definition}
\label{Def 3.1a} Let $\Omega$ and
$\left\{\Omega_\e\right\}_{\e>0}$ be open subsets of $D$. We say
that the sets $\left\{\Omega_\e\right\}_{\e>0}$ form  a
topologically admissible perturbation of $\Omega$ (shortly,
$t$-admissible), if
$\Omega_\e\,\stackrel{\mathrm{top}}{\longrightarrow}\, \Omega$ in
the sense of Definition \ref{Def 1.2}.
\end{definition}

\begin{remark}
\label{Rem 3.2} As Theorem \ref{Th 1.3} indicates, a subset
$\Omega\subset D$ admits the existence of $H^c$-admissible
perturbations if and only if $\Omega$ belongs to the family
$\mathcal{W}_w(D)$. It turns out that the assertion:
\[
\text{``}y\in W^{1,\,p}(\mathbb{R}^N),\ \Omega\in \mathcal{W}_w(D),\
\text{ and }\ \mathrm{supp}\,y\subset \overline{\Omega},\  \text{
imply }\ y\in W^{1,\,p}_0(\Omega)\text{"}
\]
is not true, in general. In particular, the above statement does
not take place in the case when an open domain $\Omega$ has a
crack. So, $\mathcal{W}_w(D)$ is a rather general class of open
subsets of $D$.
\end{remark}
\begin{remark}
\label{Rem 3.2a} The remark above motivates us to say that we call
$\Omega\subset D$ a $p$-stable domain if for any $y\in
W^{1,\,p}(\mathbb{R}^N)$ such that $y=0$ almost everywhere on
$\text{int}\,\Omega^c$, we get $\left.y\right|_{\,\Omega}\in
W^{1,\,p}_0(\Omega)$. Note that this property holds for all
reasonably regular domains such as Lipschitz domains for instance.
A more precise discussion of this property may be found in
\cite{Dancer}.
\end{remark}

We begin with the following result.
\begin{proposition}
\label{Prop 3.3} Let $\Omega\in \mathcal{W}_w(D)$ be a fixed
subdomain of $D$, and let $\left\{\Omega_\e\right\}_{\e>0}$ be an
$H^c$-admissible perturbation of $\Omega$. Let
$\left\{(\mathcal{U}_{\e},y_\e,z_\e)\in
\Xi_{\Omega_\e}\right\}_{\e>0}$ be a sequence of admissible triplets
for the problems \eqref{3.1}--\eqref{3.4}. Then the sequence
$\left\{(\mathcal{U}_{\e},\widetilde{y}_{\e},\widetilde{z}_\e)\right\}_{\e>0}$
is uniformly bounded in $L^\infty(D;\mathbb{R}^{N\times N})\times
W^{1,\,p}_0(D)\times L^p(D)$ and for each its $\tau$-cluster triplet
$
(\mathcal{U}^\ast,y^\ast,z^\ast)\in L^\infty(D;\mathbb{R}^{N\times N})\times
W^{1,\,p}_0(D)\times L^p(D)
$ (e.g. a closure point for  $\tau$-topology),
we have
\begin{gather}
\label{3.4*} \mathcal{U}^\ast\in U_{ad},\\[1ex]
\label{3.5}
\begin{split}
\int_D \left(\mathcal{U}^\ast[(\nabla y^\ast)^{p-2}]\nabla y^\ast, \nabla\widetilde{\varphi}\right)_{\mathbb{R}^N}\,dx& + \int_D
|y^\ast|^{p-2}y^\ast \widetilde{\varphi}\,dx\\
&=\langle f,
\widetilde{\varphi}\rangle_{W_0^{1,p}(D)},\,\forall\,\varphi\in
C^\infty_0(\Omega),
\end{split}\\[1ex]
\label{3.5*} \int_D z^\ast \widetilde{\psi}\,dx +\langle BF(y^\ast,z^\ast),\widetilde{\psi}\rangle_{L^q(D)}=\int_D g\,
\widetilde{\psi}\,dx,\quad \forall\,\psi\in
C^\infty_0(\Omega).
\end{gather}
\end{proposition}
\begin{proof}  Since each of the triplets $(\mathcal{U}_{\e},y_{\e},z_\e)$ is
admissible to the corresponding problem \eqref{3.1}--\eqref{3.4},
the uniform boundedness of the sequence
$\left\{(\mathcal{U}_{\e},\widetilde{y}_{\e},\widetilde{z}_\e)\right\}_{\e>0}$ with respect
to the norm of $L^\infty(D;\mathbb{R}^{N\times N})\times
W^{1,\,p}_0(D)\times L^p(D)$ is a direct consequence of \eqref{2.5},  Proposition \ref{Prop
1.16},  and Theorem \ref{Th 2.8}. So, we may assume that there
exists a triplet
$(\mathcal{U}^\ast,y^\ast,z^\ast)$ such that (within a
subsequence still denoted by suffix $\e$)
$$
(\mathcal{U}_{\e},\widetilde{y}_{\e},\widetilde{z_\e})\,\stackrel{\tau}{\longrightarrow}\,(\mathcal{U}^\ast,y^\ast,z^\ast)\ \text{ in }\ L^\infty(D;\mathbb{R}^{N\times N})\times W^{1,\,p}_0(D)\times L^p(D).
$$
Then, in view of Remark \ref{rem 1.8}, we have
$\mathcal{U}^\ast\in U_{ad}$.

Let us take as test functions $\varphi\in C^\infty_0(\Omega)$ and $\psi\in C^\infty_0(\Omega)$.
Since $\Omega_\e\,\stackrel{H^c}{\longrightarrow}\,\Omega$, then
by Theorem \ref{Th 1.3}, the Sobolev spaces
$\left\{W^{1,\,p}_0(\Omega_\e)\right\}_{\e>0}$ converge in the
sense of Mosco to $W^{1,\,p}_0(\Omega)$. Hence, for the functions
$\varphi,\psi\in W^{1,\,p}_0(\Omega)$ fixed before, there exist
sequences $\left\{\varphi_\e\in
W^{1,\,p}_0(\Omega_\e)\right\}_{\e>0}$ and $\left\{\psi_\e\in
W^{1,\,p}_0(\Omega_\e)\right\}_{\e>0}$ such that
$\widetilde{\varphi}_\e\rightarrow \widetilde{\varphi}$ and $\widetilde{\psi}_\e\rightarrow \widetilde{\psi}$ strongly
in $W^{1,\,p}(D)$    (see property ($M_1$)). Since $(\mathcal{U}_{\e},y_\e,z_\e)$
is an admissible triplet for the corresponding problem in
$\Omega_\e$, we can write for every $\e>0$
\begin{gather*}
\int_{\Omega_\e}{\left(\mathcal{U}_\e[(\nabla
{y_\e})^{p-2}]\nabla y_\e,\nabla
\varphi_\e\right)_{\mathbb{R}^N}\,dx} +\int_{\Omega_\e}
|y_\e|^{p-2}y_\e\,\varphi_\e\,dx=
\langle f,\varphi_\e\rangle_{W_0^{1,p}(\Omega_\e)},\\
\int_{\Omega_\e} z_\e \psi_\e\,dx +\langle BF(y_\e,z_\e),\psi_\e\rangle_{L^q(\Omega_\e)}=\int_{\Omega_\e} g\,
\psi_\e\,dx,
\end{gather*}
and, hence,
\begin{gather}
\label{3.6}
\int_{D}{\left(\mathcal{U}_\e[(\nabla
{\widetilde{y}_\e})^{p-2}]\nabla \widetilde{y}_\e,\nabla
\widetilde{\varphi}_\e\right)_{\mathbb{R}^N}\,dx} +\int_{D}
|\widetilde{y}_\e|^{p-2}\widetilde{y}_\e\,\widetilde{\varphi}_\e\,dx=
\langle f,\widetilde{\varphi}_\e\rangle_{W_0^{1,p}(D)},
\\
\label{3.6*}
\int_D \widetilde{z}_\e\widetilde{\psi}_\e\,dx +\langle BF(\widetilde{y}_\e,\widetilde{z_\e}),\widetilde{\psi}_\e\rangle_{L^q(D)}=\int_D g\,
\widetilde{\psi}_\e\,dx.
\end{gather}

To prove the equalities \eqref{3.5}--\eqref{3.5*}, we pass to the limit in the
integral identities \eqref{3.6}--\eqref{3.6*} as $\e\to 0$. Using the arguments from \cite{OKogut2010, Kupenko2011} and Theorem \ref{Th 2.8}, we have
\begin{align*}
\mathrm{div}\,{u}_{i\,\e}\rightarrow\mathrm{div}\,u^\ast_{i}
\ &\text{ strongly in }\ W^{-1,\,q}(D),\ \forall\, i=1,\dots,n,\\
\left\{[(\nabla \widetilde{y}_\e)^{p-2}]\nabla
\widetilde{y}_\e\right\}_{\e>0}
\ &\text{ is bounded in }\ {L}^q(D;\mathbb{R}^N),\\
\left\{|\widetilde{y}_\e|^{p-2} \widetilde{y}_\e\right\}_{\e>0} \
&\text{ is bounded in }\ L^q(D),\\
\left\{\widetilde{z}_\e \right\}_{\e>0} \
\text{ is
bounded in }\ L^p(D),\,&\left\{F(\widetilde{y}_\e,\widetilde{z}_\e)\right\}_{\e>0} \
\text{ is
bounded in }\ L^p(D),\\
\widetilde{y}_\e\rightarrow y^\ast\ \text{ in }\ L^p(D),\quad
&\widetilde{y}_\e(x)\rightarrow y^\ast(x)\ \text{a.e.}\ x\in
D,\\
|\widetilde{y}_\e|^{p-2} \widetilde{y}_\e\rightarrow
|y^\ast|^{p-2} y^\ast \text{ weakly}& \text{  in } L^q(D),\,\widetilde{z}_\e\rightarrow z^\ast \text{ weakly in } L^p(D), \\
\exists \,\nu\in L^q(D)\ \text{ such that }\  &F(\widetilde{y}_\e,\widetilde{z}_\e)\rightarrow \nu\text{ weakly in } L^p(D),
\end{align*}
where $\mathcal{U}_\e=\left[u_{1\,\e},\dots,u_{N\,\e}\right]$ and
$\mathcal{U}^\ast=\left[u^\ast_{1},\dots,u^\ast_{N}\right]$.

As for the sequence $
\left\{f_\e:=f-|\widetilde{y}_\e|^{p-2}\widetilde{y}_\e\right\}_{\e>
0}$, it is clear that
\begin{equation*}
f_\e\rightarrow f_0=f-|y^\ast|^{p-2}y^\ast\quad\text{ strongly
in }\quad W^{-1,\,\,q}(D).
\end{equation*}
In view of these observations and a priori estimate \eqref{1.17}, it is easy to see that the sequence
$\left\{\mathcal{U}_\e[(\nabla \widetilde{y}_\e)^{p-2}]\nabla \widetilde{y}_\e \right\}_{\e>0}$ is bounded
in ${L}^q(D;\mathbb{R}^N)$. So, up to a subsequence, we may suppose that
there exists a vector-function $\xi\in {L}^q(D;\mathbb{R}^N)$ such
that
\begin{equation*}
\mathcal{U}_\e[(\nabla \widetilde{y}_\e)^{p-2}]\nabla \widetilde{y}_\e\to {\xi}\quad\text{ weakly in }\
{L}^q(D;\mathbb{R}^N).
\end{equation*}
As a result, using the strong convergence
$\widetilde{\varphi}_\e\rightarrow \widetilde{\varphi}$ in
$W^{1,\,p}(D)$ and the strong convergence $\widetilde{\psi}_\e\rightarrow \widetilde{\psi}$ in
$L^p(D)$, the limit passage in the relations
\eqref{3.6}--\eqref{3.6*} as $\e\to 0$  gives
\begin{gather}
\label{3.7} \int_{D}\left({\xi},\nabla
\widetilde{\varphi}\right)_{\mathbb{R}^N}\,dx=
\int_{D}\left(f-|y^\ast|^{p-2}y^\ast\,\right)
\widetilde{\varphi}\,dx,\\
\label{3.7*}
\int_D z^\ast\widetilde{\psi}\,dx +\langle B\nu,\widetilde{\psi}\rangle_{L^q(D)}=\int_D g\,
\widetilde{\psi}\,dx.
\end{gather}

To conclude the proof it remains to note that the validity of equalities
\begin{align}
\label{3.6a} \xi&= \mathcal{U}^\ast[(\nabla y^\ast)^{p-2}]\nabla y^\ast,\\
\label{3.6b} \nu&= F(y^\ast,z^\ast)
\end{align}
can be established in a similar manner as in \cite{OKogut2010,Kupenko2011} and Theorem \ref{Th 2.8}.
\end{proof}

Our next intention is to prove that every $\tau$-cluster triplet
$$
(\mathcal{U}^\ast,y^\ast,z^\ast)\in L^\infty(D;\mathbb{R}^{N\times N})\times
W^{1,\,p}_0(D)\times L^p(\Omega)
$$
of the sequence
$\left\{(\mathcal{U}_{\e},y_\e,z_{\e})\in
\Xi_\e\right\}_{\e>0}$
is admissible to the original optimal control problem
\eqref{2.7}--\eqref{2.7c}. With that in mind, as follows from
\eqref{3.4}--\eqref{3.5}, we have to show that $\left.
y^\ast\right|_\Omega\in W^{1,\,p}_0(\Omega)$ and $z^\ast\in\mathcal{H}(\left.
y^\ast\right|_\Omega)$, i.e.,
$$
\int_\Omega z^\ast\psi\,dx +\langle BF(y^\ast,z^\ast),\psi\rangle_{L^q(\Omega)}=\int_\Omega g\,
\psi\,dx,\;\forall\,\psi\in W_0^{1,p}(\Omega).
$$
To this end, we give the following result (we refer to
\cite{Bucur_Treb} for the details).

\begin{lemma}
\label{Lemma 3.9} Let $\Omega, \left\{\Omega_\e\right\}_{\e>0}\in
\mathcal{W}_w(D)$, and let
$\Omega_\e\,\stackrel{H^c}{\longrightarrow}\,\Omega$ as $\e\to 0$.
Let $\mathcal{U}_0\in M_p^{\alpha,\beta}(D)$ be a fixed matrix. Then
\begin{equation}
\label{3.10} \widetilde{v}_{\,\Omega_\e,\,h}\rightarrow
\widetilde{v}_{\,\Omega,\,h}\ \text{strongly in }\
W^{1,\,p}_0(D),\quad\forall\, h\in W^{1,\,p}_0(D),
\end{equation}
where $v_{\,\Omega_\e,\,h}$ and $v_{\,\Omega,\,h}$ are the unique
weak solutions to the boundary value problems
\begin{equation}
\label{3.11} \left.
\begin{array}{c}
-\mathrm{div}\,\left(\mathcal{U}_0[(\nabla v)^{p-2}] \nabla v\right)+
|v|^{p-2}v=0 \text{ in }
\Omega_\e,\\[1ex]
v-h\in  W^{1,\,p}_0(\Omega_\e)
\end{array}
\right\}
\end{equation}
and
\begin{equation}
\label{3.12} \left.
\begin{array}{c}
-\mathrm{div}\,\left(\mathcal{U}_0[(\nabla v)^{p-2}]\right)+
|v|^{p-2}v=0\text{ in }
\Omega,\\[1ex]
v-h\in  W^{1,\,p}_0(\Omega),
\end{array}
\right\}
\end{equation}
respectively. Here, $\widetilde{v}_{\,\Omega_\e,\,h}$ and $\widetilde{v}_{\,\Omega,\,h}$ are the extensions of $v_{\,\Omega_\e,\,h}$ and $v_{\,\Omega,\,h}$ such that they
coincide with $h$ out of $\Omega_\e$ and $\Omega$, respectively.
\end{lemma}
\begin{remark}
\label{Rem 3.12a} In general, Lemma~\ref{Lemma 3.9} is not valid if $\Omega_\e\,\stackrel{\mathrm{top}}{\longrightarrow}\,
\Omega$ (for counter-examples and more comments we refer the reader to \cite{Bucur_Treb}).
\end{remark}

We are now in a position to prove the following property.
\begin{proposition}
\label{Prop 3.13} Let $\left\{(\mathcal{U}_{\e},y_\e,z_\e)\in
\Xi_{\e}\right\}_{\e>0}$ be an arbitrary sequence of admissible solutions
to the family of optimal control problems \eqref{3.1}--\eqref{3.4}, where
$\left\{\Omega_\e\right\}_{\e>0}$ is some $H^c$-admissible
perturbation of the set $\Omega\in \mathcal{W}_w(D)$. If for a
subsequence of $\left\{(\mathcal{U}_{\e},y_\e,z_\e)\in
\Xi_\e\right\}_{\e>0}$ (still denoted by the same index $\e$) we
have
$(\mathcal{U}_{\e},\widetilde{y}_\e,\widetilde{z}_\e)\,\stackrel{\tau}{\longrightarrow}\,(\mathcal{U}^\ast,y^\ast,z^\ast)$,
then
\begin{gather}
\label{3.14.1}
y^\ast=\widetilde{y}_{\,\Omega,\,\mathcal{U}^\ast},\quad
\left.z^\ast\right|_{\Omega}\in \mathcal{H}(y_{\,\Omega,\,\mathcal{U}^\ast}), \\
\label{3.14.2}
\int_\Omega z^\ast\psi\,dx +\langle BF({y}_{\,\Omega,\,\mathcal{U}^\ast},z^\ast),\psi\rangle_{L^q(\Omega)}=\int_\Omega g\,
\psi\,dx,\;\forall\,\psi\in W_0^{1,p}(\Omega),\\
\label{3.14.3}
(\mathcal{U}^\ast,\left.y^\ast\right|_{\,\Omega},\left.z^\ast\right|_{\Omega})\in \Xi_{sol},
\end{gather}
where by $y_{\,\Omega,\,\mathcal{U}^\ast}$ we denote the weak solution of
the boundary value problem \eqref{2.7a}--\eqref{2.7b} with
$\mathcal{U}=\mathcal{U}^\ast$.
\end{proposition}
\begin{proof}
To begin with, we note that, by Propositions \ref{Prop 1.16} and \ref{Prop
3.3}, we can extract a subsequence of $\left\{(\mathcal{U}_{\e},y_\e,z_\e)\in
\Xi_{\e}\right\}_{\e>0}$ (still denoted by the same index) such that
\begin{align}
\mathcal{U}_\e\rightarrow
\mathcal{U}^\ast=\left[u^\ast_{1},\dots,u^\ast_{N}\right]\in
U_{ad}\ &\text{ weakly-}\ast\ \text{ in }\
L^\infty(D;\mathbb{R}^{N\times N}),\label{3.14b}\\
\widetilde{y}_\e\rightarrow y^\ast\ &\text{ weakly in }\ W^{1,\,p}_0(D),\label{3.14a}\\
\widetilde{z}_\e\rightarrow z^\ast\ &\text{ weakly in }\ L^p(\Omega),\\
y\in W^{1,\,p}_0(\Omega),\quad &\widetilde{y}\in
W^{1,\,p}_0(D).\notag
\end{align}
Since \eqref{3.14.2}--\eqref{3.14.3} are direct consequence of \eqref{3.14.1}, we divide the proof into two steps.

\vspace*{4pt}\noindent{\bf Step 1.} We prove that $y^\ast= \widetilde{y}$. Following Bucur \&
Trebeschi \cite{Bucur_Treb}, for every $\e>0$, we consider the new
boundary value problem
\begin{equation}
\label{3.15} \left.
\begin{array}{c}
-\mathrm{div}\,\left(\mathcal{U}^\ast[(\nabla \varphi_\e)^{p-2}]
\nabla\varphi_\e\right)+ |\varphi_\e|^{p-2}\varphi_\e=0\quad \text{
in }\quad
\Omega_\e,\\[1ex]
\varphi_\e=-y^\ast\ \text{ in }\ D\setminus \Omega_\e.
\end{array}
\right\}
\end{equation}
Passing to the variational statement of \eqref{3.15}, for every $\e>0$ we have
\begin{equation}
\notag
 \int_{D}\Big(\mathcal{U}^\ast[(\nabla
{\widetilde{\varphi}_\e})^{p-2}]\nabla
\widetilde{\varphi}_\e, \nabla
\widetilde{\psi}_\e\Big)_{\mathbb{R}^N}\,dx\\
 +\int_{D}
|\widetilde{\varphi}_\e|^{p-2}\widetilde{\varphi}_\e\,\widetilde{\psi}_\e\,dx=
0,\quad \forall\,\psi\in C^\infty_0(\Omega_\e).
\label{3.16}
\end{equation}
Taking in \eqref{3.16} as the text function
$\widetilde{\psi}_\e=\widetilde{\varphi}_\e+y^\ast -
\widetilde{y}_\e$, we obtain
\begin{align}
\notag
\int_{D}\Big(\mathcal{U}^\ast[(\nabla
{\widetilde{\varphi}_\e})^{p-2}]\nabla
\widetilde{\varphi}_\e, &\nabla \left(\widetilde{\varphi}_\e+y^\ast -
\widetilde{y}_\e\right)\Big)_{\mathbb{R}^N}\,dx\\
&+\int_{D}
|\widetilde{\varphi}_\e|^{p-2}\widetilde{\varphi}_\e\,\left(\widetilde{\varphi}_\e+y^\ast
- \widetilde{y}_\e\right)\,dx= 0,\ \forall\,\e>0.
\label{3.17}
\end{align}
Let $\varphi\in W^{1,\,p}(\Omega)$ be the weak solution to the
problem
\begin{equation*}
\left.
\begin{array}{c}
-\mathrm{div}\,\left(\mathcal{U}^\ast[(\nabla \varphi)^{p-2}]
\nabla\varphi\right)+  |\varphi|^{p-2}\varphi=0\quad \text{ in
}\quad
\Omega,\\[1ex]
\varphi=-y^\ast\ \text{ in }\ D\setminus \Omega.
\end{array}
\right\}
\end{equation*}
Then by Lemma \ref{Lemma 3.9}, we have
$\widetilde{\varphi}_\e\rightarrow \widetilde{\varphi}$ strongly
in $W^{1,\,p}_0(D)$. Hence,
\begin{align*}
\nabla\widetilde{\varphi}_\e\rightarrow \nabla \widetilde{\varphi}
&\text{ strongly in } L^p(D;\mathbb{R}^N), \\
\|[(\nabla \widetilde{\varphi}_\e)^{p-2}]\nabla
\widetilde{\varphi}_\e\|^q_{{L}^q(D;\mathbb{R}^N)}=\|\nabla
\widetilde{\varphi}_\e\|^p_{{L}^p(D;\mathbb{R}^N)}&\rightarrow \|\nabla
\widetilde{\varphi}\|^p_{{L}^p(D;\mathbb{R}^N)}\\
&\qquad=\|[(\nabla
\widetilde{\varphi})^{p-2}]\nabla
\widetilde{\varphi}\|^q_{{L}^q(D;\mathbb{R}^N)},\\
\nabla\widetilde{\varphi}_\e(x)\rightarrow \nabla
\widetilde{\varphi}(x) &\text{ a.e. in } D,
\end{align*}
and
\begin{align*}
\widetilde{\varphi}_\e\rightarrow \widetilde{\varphi}\
&\text{ strongly in }\ L^p(D), \\
\|\left|\widetilde{\varphi}_\e\right|^{p-2}
\widetilde{\varphi}_\e\|^q_{L^q(D)}=\|
\widetilde{\varphi}_\e\|^p_{L^p(D)}&\rightarrow \|
\widetilde{\varphi}\|^p_{L^p(D)}=\|\left|
\widetilde{\varphi}\right|^{p-2}
\widetilde{\varphi}\|^q_{L^q(D)},\\
\widetilde{\varphi}_\e(x)\rightarrow \widetilde{\varphi}(x)\
&\text{ a.e. in } D.
\end{align*}

Since the norm convergence together with pointwise convergence
imply the strong convergence, it follows that
\begin{align*}
[(\nabla \widetilde{\varphi}_\e)^{p-2}]\nabla
\widetilde{\varphi}_\e\rightarrow [(\nabla \widetilde{\varphi})^{p-2}]\nabla
\widetilde{\varphi}\ &\text{ strongly in }\ L^q(D;\mathbb{R}^N),\\
\left|\widetilde{\varphi}_\e\right|^{p-2}
\widetilde{\varphi}_\e\rightarrow \left|
\widetilde{\varphi}\right|^{p-2}
\widetilde{\varphi}\ &\text{ strongly in }\ L^q(D),\\
\nabla\left(\widetilde{\varphi}_\e+y^\ast - \widetilde{y}_\e\right)
\rightarrow \nabla\widetilde{\varphi}\ &\text{ weakly in }\
{L}^p(D;\mathbb{R}^N)\ (\text{ see \eqref{3.14a}}),\\
\left(\widetilde{\varphi}_\e+y^\ast -
\widetilde{y}_\e\right)\rightarrow \widetilde{\varphi}\ \ &\text{
strongly in }\ L^p(D),
\end{align*}
Hence, the integral identity \eqref{3.17} contains only the
products of weakly and strongly convergent sequences. So, passing
to the limit in \eqref{3.17} as $\e$ tends to zero, we get
\begin{equation*}
\int_{D}{\left(\mathcal{U}^\ast[(\nabla
{\widetilde{\varphi}})^{p-2}]\nabla
\widetilde{\varphi},\nabla \widetilde{\varphi}\right)_{\mathbb{R}^N}\,dx}\\
+\int_{D}|\widetilde{\varphi}|^{p}\,dx= 0.
\end{equation*}
Taking into account the properties of $\mathcal{U}^\ast$
prescribed above, we can consider the left-hand side of the above equation as a $p$-th power of norm in $W_0^{1,p}(\Omega)$, which is equivalent to \eqref{0}. Hence, it implies that $\widetilde{\varphi}=0$ a.e. in
$D$. However, by definition $\widetilde{\varphi}=-y^\ast$ in
$D\setminus\Omega$. So, $y^\ast=0$ in $D\setminus\Omega$, and we
obtain the required property $y_{\,\mathcal{U}^\ast,\,\Omega}=\left.
y^\ast\right|_\Omega\in W^{1,\,p}_0(\Omega)$.

\vspace*{4pt}\noindent{\bf Step 2.} Our aim is to show that $\left.
z^\ast\right|_\Omega\in \mathcal{H}(y_{\,\mathcal{U}^\ast,\,\Omega})$. In view of \eqref{3.5*}, from Proposition \eqref{Prop 3.3}, we get
$$
\int_\Omega z^\ast {\psi}\,dx +\int_\Omega BF(y^\ast,z^\ast){\psi}\,dx=\int_\Omega g\,
{\psi}\,dx,\quad \forall\,\psi\in
C^\infty_0(\Omega).
$$
As was shown at the first step, $y^\ast=y_{\,\mathcal{U}^\ast,\,\Omega}$ on $\Omega$, and, therefore, we can rewrite the above equality in the following way
$$
\int_\Omega z^\ast {\psi}\,dx +\int_\Omega BF(y_{\,\mathcal{U}^\ast,\,\Omega},z^\ast){\psi}\,dx=\int_\Omega g\,
{\psi}\,dx,\quad \forall\,\psi\in C^\infty_0(\Omega),
$$
which implies the inclusion  $\left.
z^\ast\right|_\Omega\in \mathcal{H}(y_{\,\mathcal{U}^\ast,\,\Omega})$.
The
proof is complete.
\end{proof}

The results given above suggest us to study the asymptotic
behavior of the sequences of admissible triplets
$\left\{(\mathcal{U}_{\e},y_\e,z_\e)\in
\Xi_{\e}\right\}_{\e>0}$ for the case of $t$-admissible
perturbations of the set $\Omega$.

\begin{proposition}
\label{Prop 3.20.1} Let $\Omega$ be a $p$-stable open subset of
$D$. Let $\left\{(\mathcal{U}_{\e},y_\e,z_\e)\in
\Xi_\e\right\}_{\e>0}$ be a sequence of admissible triplets
for the family \eqref{3.1}--\eqref{3.4}, where
$\left\{\Omega_\e\right\}_{\e>0}\subset D$ form a $t$-admissible
perturbation of $\Omega$. Then
$$
\left\{(\mathcal{U}_{\e},\widetilde{y}_\e,\widetilde{z}_\e)\right\}_{\e>0}\
\text{ is uniformly bounded in }\ L^\infty(D;\mathbb{R}^{N\times
N})\times W^{1,\,p}_0(D)\times L^p(D)
$$
and for every $\tau$-cluster triplet $(\mathcal{U}^\ast,y^\ast,z^\ast)\in
L^\infty(D;\mathbb{R}^{N\times N})\times W^{1,\,p}_0(D)\times L^p(\Omega)$ of this
sequence, we have
\begin{enumerate}
\item[(j)] the triplet $(\mathcal{U}^\ast,y^\ast,z^\ast)$ satisfies the relations \eqref{3.4*}--\eqref{3.5*};

\item[(jj)] the triplet $(\mathcal{U}^\ast,\left.y^\ast\right|_{\,\Omega},\left.z^\ast\right|_{\,\Omega})$ is admissible to
the problem \eqref{2.7}--\eqref{2.7c}, i.e.,
$y^\ast=\widetilde{y}_{\,\Omega,\,\mathcal{U}^\ast}$, $\left.z^\ast\right|_{\Omega}\in \mathcal{H}({y}_{\,\Omega,\,\mathcal{U}^\ast})$, where
${y}_{\,\Omega,\,\mathcal{U}^\ast}$ stands for the weak solution of the
boundary value problem \eqref{2.7a}--\eqref{2.7b} under
$\mathcal{U}=\mathcal{U}^\ast$.
\end{enumerate}
\end{proposition}

\begin{proof}
Since $\Omega_\e\,\stackrel{\mathrm{top}}{\longrightarrow}\,
\Omega$ in the sense of Definition \ref{Def 1.2}, it follows that
for any $\varphi,\psi\in C^\infty_0(\Omega\setminus K_0)$ we have
$\mathrm{supp}\,\varphi\subset \Omega_\e$, $\mathrm{supp}\,\psi\subset \Omega_\e$ for all $\e>0$ small
enough. Moreover, since the set $K_0$ has zero $p$-capacity, it
follows that $C^\infty_0(\Omega\setminus K_0)$ is dense in
$W^{1,\,p}_0(\Omega)$. Therefore, the verification of item (j) can
be done in an analogous way to the proof of Proposition \ref{Prop
3.3} replacing therein the sequences $\left\{\varphi_\e\in
W^{1,\,p}_0(\Omega_\e)\right\}_{\e>0}$ and $\left\{\psi_\e\in
W^{1,\,p}_0(\Omega_\e)\right\}_{\e>0}$ by the still functions
$\varphi$ and $\psi$. As for the rest, we have to repeat all arguments of
that proof.

To prove the assertion (jj), it is enough to show that
$\left.y^\ast\right|_{\,\Omega}\in W^{1,\,p}_0(\Omega)$. To do so,
let $B_0$ be an arbitrary closed ball not intersecting
$\overline{\Omega}\cup K_1$. Then from \eqref{3.2}--\eqref{3.3} it
follows that $\widetilde{y}_\e=\widetilde{y}_{\,\Omega_\e,\mathcal{U}_\e}=0$
almost everywhere in $B_0$ whenever the parameter $\e$ is small
enough. Since by (j) and Sobolev Embedding Theorem
$\widetilde{y}_{\e}$ converges to $y^\ast$
strongly in $L^p(D)$, it follows that the same is true for the
limit function $y^\ast$. As the ball $B_0$ was chosen arbitrary,
and $K_1$ is of Lebesgue measure zero, it follows that
$\mathrm{supp}\, y^\ast\subset\Omega$. Then, by Fubini's Theorem,
we have  $\mathrm{supp}\, y^\ast\subset\overline{\Omega}$. Hence,
using the properties of $p$-stable domains (see Remark \ref{Rem
3.2a}), we just come to the desired conclusion:
$\left.y^\ast\right|_{\,\Omega}\in W^{1,\,p}_0(\Omega)$. The rest of the proof should be quite similar to the one of Proposition \ref{Prop 3.13}, where we showed, that $\left.z^\ast\right|_\Omega\in \mathcal{H}(\left.y^\ast\right|_{\,\Omega})$. The proof is
complete.
\end{proof}

\begin{corollary}
\label{Cor 3.18} Let $\{(\mathcal{U}_\e,y_\e,z_\e)\in \Xi_\e\}_{\e>0}$ be a sequence such that
$\mathcal{U}_\e\equiv \mathcal{U}^\ast$, $\forall\,\e>0$, where $\mathcal{U}^\ast\in U_{ad}$ is an admissible
control.  Let $\left\{y_{\,\Omega_\e,\,\mathcal{U}^\ast}\in
W^{1,\,p}_0(\Omega_\e)\right\}_{\e>0}$  be the corresponding
solutions of \eqref{3.2}--\eqref{3.3} and  let $z_\e\in \mathcal{H}(y_{\,\Omega_\e,\,\mathcal{U}^\ast})$ be any
solutions of \eqref{3.3} for each $\e>0$. Then, under assumptions of
Proposition \ref{Prop 3.13} or Proposition \ref{Prop 3.20.1}, we
have that, within a subsequence still denoted by the same index $\e$, the following convergence takes place
\begin{gather*}
\widetilde{y}_{\,\Omega_\e,\,\mathcal{U}^\ast}\rightarrow
\widetilde{y}_{\,\Omega,\,\mathcal{U}^\ast}\ \text{ strongly in }\
W^{1,\,p}_0(D),\\
\widetilde{z}_\e\rightarrow z^\ast \ \text{ strongly in }\
L^p(D), \ \text{ and }\ \left.z^\ast\right|_\Omega\in\mathcal{H}({y}_{\,\Omega,\,\mathcal{U}^\ast}).
\end{gather*}
\end{corollary}
\begin{proof}
The proof is given in Appendix.
\end{proof}

\section{Mosco-stability of optimal control problems}
\label{Sec_4}

We begin this section with the following concept.

\begin{definition}
\label{Def 4.1} We say that the optimal control problem
\eqref{2.7}--\eqref{2.7c} in $\Omega$ is Mosco-stable in
$L^\infty(D;\mathbb{R}^{N\times N})\times W^{1,\,p}_0(D)\times L^p(\Omega)$ along
the perturbation $\left\{\Omega_\e\right\}_{\e>0}$  of $\Omega$, if the following conditions are
satisfied
\begin{enumerate}
\item[(i)]
if $\left\{(\mathcal{U}^0_\e,
y^{\,0}_\e,z^{\,0}_\e)\in{\Xi}_{\e}\right\}_{\e>0}$ is a sequence of
optimal solutions to the perturbed problems
\eqref{3.1}--\eqref{3.4}, then this sequence is relatively
$\tau$-compact in $L^\infty(D;\break\mathbb{R}^{N\times N})\times
W^{1,\,p}_0(D)\times L^p(D)$;
\item[(ii)] each $\tau$-cluster triplet of $\left\{(\mathcal{U}^0_\e,
y^{\,0}_\e,z^{\,0}_\e)\in{\Xi}_{\e}\right\}_{\e>0}$ is an optimal
solution to the original problem \eqref{2.7}--\eqref{2.7c}.
\end{enumerate}
Moreover, if
\begin{equation}
\label{4.4} (\mathcal{U}^0_\e,
\widetilde{y}^{\,0}_\e,\widetilde{z}^{\,0}_\e)\,\stackrel{\tau}{\longrightarrow}\,
(\mathcal{U}^0,y^{\,0},z^0),
\end{equation}
then $(\mathcal{U}^0,\left. y^{\,0}\right|_{\Omega},\left. z^{\,0}\right|_{\Omega})\in \Xi_{sol}$ and
\begin{equation}
\label{4.5} \inf_{(\mathcal{U},\,y,\,z)\in\,\Xi_{sol}}I_{\,\Omega}(\mathcal{U},y,z)=
I_{\,\Omega}(\mathcal{U}^0,\left. y^{\,0}\right|_{\Omega},\left. z^{\,0}\right|_{\Omega})  =\lim_{\e\to
0}\inf_{(\mathcal{U}_{\e}, y_{\e},z_{\e})\in\,\Xi_{\e}}
I_{\,\Omega_\e}(\mathcal{U}_{\e}, y_{\e},z_{\e}).
\end{equation}
\end{definition}
Our next intention is to derive the sufficient conditions for the
Mosco-stability of optimal control problem
\eqref{2.7}--\eqref{2.7b}.

\begin{theorem}
\label{Th 3.21} Let $\Omega$, $\left\{\Omega_\e\right\}_{\e>0}$ be
open subsets of $D$, and let
$$
\Xi_{\e}\subset L^\infty(D;\mathbb{R}^{N\times N})\times
W^{1,\,p}_0(\Omega_\e)\ \text{ and }\ \Xi_{sol}\subset
L^\infty(D;\mathbb{R}^{N\times N})\times W^{1,\,p}_0(\Omega)
$$
be the sets of admissible solutions to optimal control
problems \eqref{3.1}--\eqref{3.4} and \eqref{2.7}--\eqref{2.7c},
respectively. Assume that
the distributions $z_d\in L^p(D)$  in the cost functional
\eqref{2.7} and $g\in L^p(D)$ in \eqref{2.7c} are such that
\begin{equation}
\label{4.13} z_d(x)=z_d(x) \chi_{\,\Omega}(x),\quad g(x)=g(x)\chi_{\,\Omega}(x)\quad\text{ for a.e.
}\ x\in D.
\end{equation}
Assume also that Hammerstein equation \eqref{3.4.1}  possesses property $(\mathfrak{B})$ and
at least one of the suppositions
\begin{enumerate}
\item[1.] $\Omega\in \mathcal{W}_w(D)$ and $\left\{\Omega_\e\right\}_{\e>0}$ is an $H^c$-admissible
perturbation of $\Omega$;

\item[2.] $\Omega$ is a $p$-stable domain and $\left\{\Omega_\e\right\}_{\e>0}$ is a $t$-admissible
perturbation of $\Omega$;
\end{enumerate}
holds true.

Then the following assertions are valid:
\begin{enumerate}
\item[$(MS_1)$] if $\left\{\e_k\right\}_{k\in \mathbb{N}}$ is a
numerical sequence converging to $0$, and
$\left\{(\mathcal{U}_k,y_k,z_k)\right\}_{k\in \mathbb{N}}$ is a sequence
satisfying
\begin{gather*}
(\mathcal{U}_k,y_k,z_k)\in \Xi_{\e_k},\quad\forall\, k\in \mathbb{N},\
\text{ and
}\\
(\mathcal{U}_k,\widetilde{y}_k,\widetilde{z}_k)\,\stackrel{\tau}{\longrightarrow}\,
(\mathcal{U},\psi,\xi)\ \text{  in }\ L^\infty(D;\mathbb{R}^{N\times N})\times
W^{1,\,p}_0(D)\times L^p(D),
\end{gather*}
then there exist functions $y\in W^{1,\,p}_0(\Omega)$ and $z\in L^p(\Omega)$ such that
$y=\left.\psi\right|_{\Omega}$, $z=\left.\xi\right|_{\Omega}$, $z\in\mathcal{H}(y)$, $(\mathcal{U},y,z)\in\Xi_{\Omega}$, and
$$
\liminf_{k\to\infty}I_{\,\Omega_{\e_k}}(\mathcal{U}_k,y_k,z_k)\ge
I_{\,\Omega}(\mathcal{U},\left.y\right|_{\Omega},\left.z\right|_{\Omega});
$$

\item[$(MS_2)$] for any admissible triplet $(\mathcal{U},y,z)\in\Xi_{sol}$,  there exists  a realizing
sequence $\left\{(\mathcal{U}_\e,y_\e,z_\e)\in \Xi_{\e}\right\}_{\e>0}$
such that
\begin{align*}
\mathcal{U}_\e\rightarrow \mathcal{U}&\mbox{ strongly in }
L^\infty(D;\mathbb{R}^{N\times N}),\\
\widetilde{y}_\e\rightarrow \widetilde{y}&\mbox{ strongly in }
W^{1,\,p}_0(D),\\
\widetilde{z_\e}\to \widetilde{z}&\mbox{ strongly in }L^p(D),\\
\limsup_{\e\to 0}I_{\,\Omega_{\e}}&(\mathcal{U}_\e,y_\e,z_\e)\le
I_{\,\Omega}(\mathcal{U},y,z).
\end{align*}

\end{enumerate}
\end{theorem}

\begin{proof}
To begin with, we note that the first part of property ($MS_1$) is the direct
consequence of Propositions \ref{Prop 3.13} and \ref{Prop 3.20.1}.
So, it remains to check the corresponding property for cost functionals. Indeed, since $z_k\to z$ weakly in $L^p(D)$, in view of lower weak semicontinuity of norm in $L^p(D)$, we have
\begin{align*}
\liminf_{k\to\infty}I_{\,\Omega_{\e_k}}(\mathcal{U}_k,y_k,z_k)&=
\liminf_{k\to\infty}\int_{D}
|\widetilde{z}_{k}-z_d|^p\,dx
\ge \int_{D} |z-z_d|^p\,dx \\
&\ge
 \int_{\Omega} |z-z_d|^p\,dx=\int_{\Omega} \left|\left.z\right|_{\Omega}-z_d\right|^p\,dx =I_{\,\Omega}(\mathcal{U},\left.
y\right|_{\,\Omega},\left.
z\right|_{\,\Omega}).
\end{align*}
Hence, the assertion ($MS_1$) holds true.

Further, we prove ($MS_2$). In view of our initial assumptions, the set of admissible pairs
$\Xi_{sol}$ to the problem \eqref{2.7}--\eqref{2.7c} is nonempty. Let $(\mathcal{U},y,z)\in\Xi_{sol}$ be an admissible triplet.
Since the matrix $\mathcal{U}$ is an admissible
control to the problem \eqref{3.1}--\eqref{3.4} for every $\e>0$, we construct
the sequence $\left\{(\mathcal{U}_\e,y_\e,z_\e)\in \Xi_{\e}\right\}_{\e>0}$
 as follows: $\mathcal{U}_\e=\mathcal{U}$,
$\forall\,\e>0$ and $y_\e=y_{\,\Omega_\e,\mathcal{U}}$ is the corresponding
solution of boundary value problem \eqref{3.2}--\eqref{3.3}. As for the choice of elements $z_\e$, we make it later on.

Then, by Corollary \ref{Cor 3.18}, we have
\begin{equation*}
\widetilde{y}_{\,\Omega_\e,\,\mathcal{U}}\rightarrow
\widetilde{y}_{\,\Omega,\,\mathcal{U}}\ \text{ strongly in }\
W^{1,\,p}_0(D),
\end{equation*}
where $y_{\,\Omega,\,\mathcal{U}}$ is a unique solution for
\eqref{2.7a}--\eqref{2.7b}. Then the inclusion $(\mathcal{U},y,z)\in\Xi_{sol}$ implies $y=y_{\,\Omega,\,\mathcal{U}}$.

By the initial assumptions $g(x)=g(x)\chi_\Omega(x)$. Hence,
$$
\int_D \widetilde z\psi\,dx+\int_D BF(\widetilde{y},\widetilde{z})\psi\,dx=\int_D g\psi\,dx,\;\forall\,\psi\in C_0^\infty(D),
$$
i.e. $\widetilde{z}\in\mathcal{H}(\widetilde{y})\subset L^p(D)$.
Then, in view of $(\mathfrak{B})$-property, for the given pair $(\widetilde{y},\widetilde{z})$  there exists a sequence $\{\widehat{z}_\e\in \mathcal{H}(\widetilde{y}_{\,\Omega_\e,\,\mathcal{U}})\}_{\e>0}$ such that $\widehat{z}_\e\to \widetilde{z}$ strongly in $L^p(\Omega)$. As a result, we can take $\{(\mathcal{U}_\e,\widetilde{y}_\e,\widehat{z}_\e)\}$ as a realizing sequence. Moreover, in this case the desired property of the cost functional seems pretty obvious. Indeed,
\begin{align*}
\limsup_{\e\to 0}I_{\,\Omega_{\e}}(\mathcal{U}_e,y_e,z_e)&=
\limsup_{\e\to 0}\int_{D}
|\widehat{z}_\e-z_d|^p\,dx
= \int_{D} |\widetilde{z}-z_d|^p\,dx \\
&=
 \int_{\Omega} |z-z_d|^p\,dx=I_{\,\Omega}(\mathcal{U},y,z).
\end{align*}

The proof is complete.
\end{proof}


\begin{theorem}
\label{Th 4.12} Under the assumptions of Theorem \ref{Th 3.21} the optimal control problem
\eqref{2.7}--\eqref{2.7c} is Mosco-stable in
$L^\infty(D;\mathbb{R}^{N\times N})\times W^{1,\,p}_0(D)\times L^p(D)$.
\end{theorem}

\begin{proof}
In view of a priory estimates \eqref{1.3}, \eqref{1.17} and \eqref{7*}, we can
immediately conclude that any sequence of optimal pairs
$\left\{(\mathcal{U}^0_\e, y^{\,0}_\e,z^{\,0}_\e)\in{\Xi}_{\e}\right\}_{\e>0}$
to the perturbed problems \eqref{3.1}--\eqref{3.4} is uniformly bounded and, hence, relatively
$\tau$-compact in $L^\infty\break(D;\mathbb{R}^{N\times N})\times
W^{1,\,p}_0(D)\times L^p(\Omega)$. So, we may suppose that there exist a subsequence
$\left\{(\mathcal{U}^0_{\e_k}, y^{\,0}_{\e_k},z^{\,0}_{\e_k})\right\}_{\,k\in\,\mathbb{N}}$
and a triplet $(\mathcal{U}^*,y^*,z^*)$ such that $(\mathcal{U}^0_{\e_k},
\widetilde{y}^{\,0}_{\e_k},\widetilde{z}^{\,0}_{\e_k})\,\stackrel{\tau}{\longrightarrow}\,
(\mathcal{U}^*,y^*,z^*)$ as $k\to \infty$. Then, by  Theorem \ref{Th 3.21} (see property
$(MS_1)$), we have
$(\mathcal{U}^*,\left.y^*\right|_{\Omega},\left.z^*\right|_{\Omega})\in\Xi_{sol}$ and
\begin{align}
\notag
    \liminf_{k\to\infty}\min_{(\mathcal{U},\,y,\,z)\in\,\Xi_{\e_k}}
    I_{\Omega_{\e_k}}(\mathcal{U},y,z)&=
    \liminf_{k\to\infty}
    I_{\Omega_{\e_k}}(\mathcal{U}^0_{\e_k},
y^{\,0}_{\e_k},z^{\,0}_{\e_k})\\
\notag
&
\ge I_{\Omega}(\mathcal{U}^*,\left.y^*\right|_{\Omega},\left.z^*\right|_{\Omega})\\
\label{4.6}
&\ge
\min_{(\mathcal{U},\,y,\,z)\in\,\Xi_{sol}}I_{\,\Omega}(\mathcal{U},y,z) =
I_{\,\Omega}(\mathcal{U}^{opt},y^{opt},z^{opt}).
\end{align}
However, condition ($MS_2$) implies that for the optimal triplet $(\mathcal{U}^{opt},y^{opt},z^{opt})\in\Xi_{sol}$ there exists a realizing sequence
$\left\{(\widehat{\mathcal{U}}_\e,\widehat{y}_\e,\widehat{z}_\e)
\in\Xi_{\e}\right\}_{\e>0}$ such that
\begin{gather*}
(\widehat{\mathcal{U}}_\e,\widetilde{\widehat{y}}_\e,\widetilde{\widehat{z}}_\e)
\rightarrow
(\mathcal{U}^{opt},\widetilde{y}^{opt},\widetilde{z}^{opt}),\  \text{and}\\
I_{\,\Omega}(\mathcal{U}^{opt},y^{opt},z^{opt})\ge \limsup_{\e\to 0}
I_{\,\Omega_\e}(\widehat{\mathcal{U}}_\e,\widehat{y}_\e,\widehat{z}_\e).
\end{gather*}
Using this fact, we have
\begin{align}
\notag
    \min_{(\mathcal{U},\,y,\,z)\in\,\Xi_{sol}}I_{\,\Omega}(\mathcal{U},y,z) &=
    I_{\,\Omega}(\mathcal{U}^{opt},y^{opt},z^{opt})
    \ge \limsup_{\e\to 0}
I_{\,\Omega_\e}(\widehat{\mathcal{U}}_\e,\widehat{y}_\e,\widehat{z}_\e)\\
\notag
&\ge
\limsup_{\e\to
    0}\min_{(\mathcal{U},\,y,\,z)\,\in\Xi_{\e}} I_{\,\Omega_\e}(\mathcal{U},y,z)\\
    \notag
    &\ge \limsup_{k\to\infty}
    \min_{(\mathcal{U},\,y,\,z)\in\,\Xi_{\e_k}}
    I_{\Omega_{\e_k}}(\mathcal{U},y,z)\\
\label{4.7}
    &=
    \limsup_{k\to\infty}I_{\Omega_{\e_k}}(\mathcal{U}^0_{\e_k},
y^{\,0}_{\e_k},z^{\,0}_{\e_k}).
\end{align}
From this and \eqref{4.6}, we deduce
$$
    \liminf_{k\to\infty}
    I_{\Omega_{\e_k}}(\mathcal{U}^0_{\e_k},
y^{\,0}_{\e_k},z^{\,0}_{\e_k})\ge
\limsup_{k\to\infty}I_{\Omega_{\e_k}}(\mathcal{U}^0_{\e_k}, y^{\,0}_{\e_k},z^{\,0}_{\e_k}).
$$
Thus, combining the relations \eqref{4.6} and \eqref{4.7}, and
rewriting them in the form of equalities, we finally obtain
\begin{align}
\label{4.8}
    I_{\Omega}(\mathcal{U}^*,\left.y^*\right|_{\Omega},\left.z^*\right|_{\Omega})&=I_{\,\Omega}(\mathcal{U}^{opt},y^{opt},z^{opt})=
    \min_{(\mathcal{U},\,y,\,z)\in\,\Xi_{sol}}I_{\,\Omega}(\mathcal{U},y,z),\\
\label{4.9}
    I_{\,\Omega}(\mathcal{U}^{opt},y^{opt},z^{opt})&=\lim_{k\to\infty}\min_{(\mathcal{U},\,y,\,z)\in\,\Xi_{\e_k}}
    I_{\Omega_{\e_k}}(\mathcal{U},y,z).
\end{align}
Since equalities \eqref{4.8}--\eqref{4.9} hold true for every
$\tau$-convergent subsequence of the original sequence of optimal
solutions $\left\{(\mathcal{U}^0_\e,
y^{\,0}_\e,z^{\,0}_\e)\in{\Xi}_{\e}\right\}_{\e>0}$, it follows that
the limits in \eqref{4.8}--\eqref{4.9} coincide and, therefore,
$I_{\,\Omega}(\mathcal{U}^{opt},y^{opt},z^{opt})$ is the limit of the whole sequence
of minimal values $ \left\{I_{\Omega_{\e}}(\mathcal{U}^0_{\e}, y^{\,0}_{\e},z^{\,0}_\e)
=\inf_{(\mathcal{U},y,z)\in\,\Xi_{\e}}I_{\,\Omega_\e}(\mathcal{U},y,z)\right\}_{\e>0}$. This concludes the proof.
\end{proof}

\begin{remark}
It is worth to emphasize that without $(\mathfrak{B})$-property, the original optimal control problem can lose the Mosco-stability property with respect to the given type of domain perturbations. In such case there is no guarantee that each of optimal triplets to the OCP \eqref{2.7}--\eqref{2.7c} can be attained  through some sequence of optimal triplets to the perturbed problems \eqref{3.1}--\eqref{3.4}.
\end{remark}
\begin{remark} It is a principle point of our consideration, that we deal with the BVP for coupled Hammerstein-type system with Dirichlet boundary conditions. The question about stability of the similar OCP with Neumann boundary conditions remains open. In the meantime, this approach can be easily extended to the case when the boundary $\partial\Omega$ can be split onto two disjoint parts $\Gamma_1$ and $\Gamma_2$ with Dirichlet conditions on $\Gamma_1$ and Neumann conditions on $\Gamma_2$. In this case for the considered differential equation as a solution space it is enough to take instead of $W_0^{1,p}(\Omega)$ the following space
$$
W(\Omega;\Gamma_1)=cl_{\|\cdot\|_{W^{1,p}(\Omega)}}\{C_0^\infty(\Omega;\Gamma_1)\},
$$
 where $C_0^\infty(\Omega;\Gamma_1)=\{\varphi\in C_0^\infty(\mathbb{R}^N):\, \varphi|_{\Gamma_1}=0\}$.
\end{remark}

\section{Appendix}
\begin{remark}\label{Rem Ap.1}
Here we give examples to the fact that without supplementary regularity assumptions on the
sets, there is no connection between topological set convergence
and the set convergence in the Hausdorff complementary topology.
Indeed, the topological set convergence allows certain parts of
the subsets $\Omega_\e$ degenerating and being deleted in the
limit. For instance, assume that $\Omega$ consists of two disjoint
balls, and $\Omega_\e$ is a dumbbell with a small hole on each
side. Shrinking the holes and the handle, we can approximate the
set $\Omega$ by sets $\Omega_\e$ in the sense of Definition
\ref{Def 1.2} as shown in Figure \ref{Fig 1.1}.
\begin{figure}[htbp]
\begin{center}
\includegraphics[width=8cm]{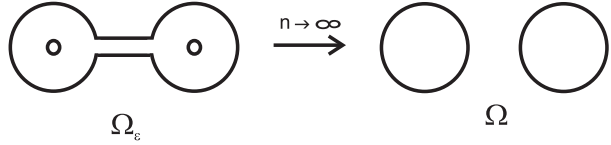}
\end{center}
\caption{Example of the set convergence in the sense of Definition
\ref{Def 1.2}} \label{Fig 1.1}
\end{figure}
It is obvious that in this case $d_{H^c}(\Omega_\e,\Omega)$ does
not converge to $0$ as $\e\to 0$. However, as an estimate of an
 \textquotedblleft approximation\textquotedblright\ of $\Omega$ by elements of the above sequence
$\Omega_\e\,\stackrel{\mathrm{top}}{\longrightarrow}\, \Omega$, we
can take the Lebesgue measure of the symmetric set difference
$\Omega_\e\triangle \Omega$, that is,
$\mu(\Omega,\Omega_\e)=\mathcal{L}^N(\Omega\setminus\Omega_\e\cup\Omega_\e\setminus\Omega)$.
It should be noted that in this case the distance $\mu$ coincides
with the well-known Ekeland metric in $L^\infty(D)$ applied to
characteristic functions:
\[
d_E(\chi_{\,\Omega}, \chi_{\,\Omega_\e})=\mathcal{L}^N\left\{x\in
D\,:\ \chi_{\,\Omega}(x) \neq \chi_{\,\Omega_\e}(x)\right\}=
\mu(\Omega,\Omega_\e).
\]
As an example of subsets which are $H^c$-convergent but
have no limit in the sense of Definition \ref{Def 1.2}, let us
consider the sets $\left\{\Omega_\e\right\}_{\e>0}$ containing an
oscillating crack with vanishing amplitude $\e$ (see Figure
\ref{Fig 1.2}).

\begin{figure}[htbp]
\begin{center}
\includegraphics[width=8cm]{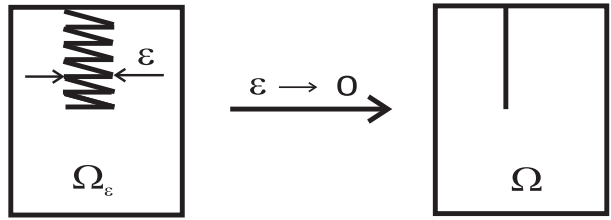}
\end{center}
\caption{The $p$-unstable sets which are compact with respect to
the $H^c$-topology} \label{Fig 1.2}
\end{figure}
\end{remark}

\subsection{Proof of Proposition \ref{prop 1.15}}

Let $\mathcal{U}\in U_{ad}$ be an arbitrary admissible control. Then for a given $f\in W^{-1,q}(D)$, the Dirichlet boundary problem \eqref{2.7a}--\eqref{2.7b} admits a unique solution $y_\mathcal{U}=y(\mathcal{U},f)\in W_0^{1,p}$ for which the estimate \eqref{1.17} holds true. It remains to remark that the corresponding Hammerstein equation
\begin{equation}\label{ast}
z+BF(y_\mathcal{U},z)=g
\end{equation}
has a nonempty set of solutions $\mathcal{H}(y_\mathcal{U})$ for every $g\in L^p(D)$ by Theorem \ref{Th 1.1*}.

\subsection{Proof of Corollary \ref{Rem 1.8}}

Let $\left\{ y_k\right\}_{k\in \mathbb{N}}$  be a given sequence, and let $y_0\in W_0^{1,p}(\Omega)$ be its strong limit. Let $\left\{ z_k\in \mathcal{H}(y_k)\right\}_{k\in \mathbb{N}}$ be an arbitrary sequence of corresponding solutions to the Hammerstein equation \eqref{1.9.2}.
As follows from the proof of Theorem \ref{Th 2.8}, the sequence $\{z_k\in \mathcal{H}(y_k)\}_{k\in\mathbb{N}}$ is uniformly bounded in $L^p(\Omega)$ and, moreover, there exist a subsequence of $\{z_k\}_{k\in\mathbb{N}}$  still denoted by the same index and an element $z_0\in L^p(\Omega)$ such that $z_k\to z_0$ weakly in $L^p(\Omega)$ and $z_0\in \mathcal{H}(y_0)$. Our aim is to show that in this case $z_k\to z_0$ strongly in $L^p(\Omega)$. Indeed, as follows from \eqref{1.22} and \eqref{3*}, we have the following equalities
\begin{align}
\label{4*}
\langle F(y_k,z_k),z_k\rangle_{L^p(\Omega)}+\langle F(y_k,z_k), BF(y_k,z_k)\rangle_{L^p(\Omega)}&=
\langle F(y_k,z_k),g\rangle_{L^p(\Omega)},\;\forall k\in\mathbb{N},\\
\label{5*}
\langle F(y_0,z_0),z_0\rangle_{L^p(\Omega)}+\langle F(y_0,z_0), BF(y_0,z_0)\rangle_{L^p(\Omega)}&=
\langle F(y_0,z_0),g\rangle_{L^p(\Omega)}.
\end{align}
Taking into account that $F(y_k,z_k)\to F(y_0,z_0)$ weakly in $L^q(\Omega)$ (see Theorem \ref{Th 2.8}), the limit passage in \eqref{4*} leads us to the relation
\begin{equation}\label{6*}
\lim_{k\to\infty}\left(\langle F(y_k,z_k),z_k\rangle_{L^p(\Omega)}+\langle F(y_k,z_k), BF(y_k,z_k)\rangle_{L^p(\Omega)}\right)=\langle F(y_0,z_0), g\rangle_{L^p(\Omega)}.
\end{equation}
Since the right-hand sides  of \eqref{5*} and \eqref{6*} coincide, the lower semicontinuity of the functional $\langle Bv,v\rangle_{L^p(\Omega)}$ with respect to the weak topology of $L^p(\Omega)$ and $(\mathfrak{A})$-property of operator $F:W_0^{1,p}(\Omega)\times L^p(\Omega)\to L^q(\Omega)$ imply
\begin{align*}
\langle F(y_0,z_0),z_0\rangle_{L^p(\Omega)}&+\langle F(y_0,z_0), BF(y_0,z_0)\rangle_{L^p(\Omega)}=\langle F(y_0,z_0), g\rangle_{L^p(\Omega)}\\
&=
 \lim_{k\to\infty}\Big[\langle F(y_k,z_k),z_k\rangle_{L^p(\Omega)}+\langle F(y_k,z_k), BF(y_k,z_k)\rangle_{L^p(\Omega)}\Big]\\
 &\ge \liminf_{k\to\infty}\Big[\langle  F(y_k,z_k),z_k\rangle_{L^p(\Omega)}+\langle F(y_k,z_k), BF(y_k,z_k)\rangle_{L^p(\Omega)}\Big]\\
&\ge\langle  F(y_0,z_0),z_0\rangle_{L^p(\Omega)}+\langle F(y_0,z_0), BF(y_0,z_0)\rangle_{L^p(\Omega)}.
\end{align*}
Hence,
\begin{align*}
\lim_{k\to\infty}\langle F(y_k,z_k),z_k\rangle_{L^p(\Omega)}&=\langle  F(y_0,z_0),z_0\rangle_{L^p(\Omega)},\\
\lim_{k\to\infty}\langle F(y_k,z_k), BF(y_k,z_k)\rangle_{L^p(\Omega)} &=\langle F(y_0,z_0), BF(y_0,z_0)\rangle_{L^p(\Omega)}.
\end{align*}
To conclude the proof, it remains to apply the $(\mathfrak{M})$-property of operator $F:W_0^{1,p}(\Omega)\times L^p(\Omega)\to L^q(\Omega)$.

\begin{remark}\label{Rem 1.9}
It is worth to emphasize that
Corollary \ref{Rem 1.8} leads to the following important property of Hammerstein equation \eqref{2.7c}: if the operator $F:W_0^{1,p}(\Omega)\times L^p(\Omega)\to L^q(\Omega)$ is compact and possesses $(\mathfrak{M})$ and $(\mathfrak{A})$ properties, then the solution set $\mathcal{H}(y)$ of \eqref{2.7c} is compact with respect to the strong topology in $L^p(\Omega)$ for every element $y\in W_0^{1,p}(\Omega)$. Indeed, the validity of this assertion immediately follows from Corollary \ref{Rem 1.8} if we apply it to the sequence $\{y_k\equiv y\}_{k\in\mathbb{N}}$ and make use of the weak compactness property of $\mathcal{H}(y)$.
\end{remark}
\begin{remark} \label{Rem 1.10} As an example of the nonlinear operator $F:W_0^{1,p}(\Omega)\times L^p(\Omega)\to L^q(\Omega)$ satisfying all conditions of Theorem \ref{Th 2.8} and Corollary \ref{rem 1.8}, we can consider the following one
$$
F(y,z)=|y|^{p-2}y+|z|^{p-2}z.
$$
Indeed, this function is obviously radially continuous and  it is also strictly monotone
\begin{align*}
\langle F(y,z_1) -  F(y,z_2),z_1-z_2\rangle_{L^p(\Omega)}&=\int_{\Omega}\left(|z_1|^{p-2}z_1-|z_2|^{p-2}z_2\right)(z_1-z_2)\,dx\\
&\ge 2^{2-p}\|z_1-z_2\|^p_{L^p(\Omega)}\ge 0.
\end{align*}
This implies that $F$ is an operator with u.s.b.v.
It is also easy to see that $F$ is compact with respect to the first argument. Indeed, if $y_k\to y$ weakly in $W_0^{1,p}(\Omega)$, then, in view of the Sobolev embedding theorem, we have $y_k\to y$ strongly in $L^p(\Omega)$. Combining this fact with the convergence of norms
$$
\|\left|y_k\right|^{p-2}
y_k\|^q_{L^q(\Omega)}=\|
y_k\|^p_{L^p(\Omega)}\rightarrow \|
y\|^p_{L^p(\Omega)}=\|\left|
y\right|^{p-2}
y\|^q_{L^q(\Omega)}
$$
we arrive at the strong convergence $|y_k|^{p-2}y_k\to|y|^{p-2}y$ in $L^q(\Omega)$. As a result, we have $F(y_k,z)\to F(y,z)$ strongly in $L^q(\Omega)$.

Let us show that $F$ possesses the $(\mathfrak{M})$ and $(\mathfrak{A})$ properties. As for the $(\mathfrak{M})$ property, let $y_k\to y$ strongly in $W_0^{1,p}(\Omega)$ and $z_k\to z$ weakly in $L^p(\Omega)$ and the following condition holds
$$
\lim_{k\to\infty}\langle F(y_k,z_k),z_k\rangle_{L^p(\Omega)}=\langle F(y,z),z\rangle_{L^p(\Omega)}.
$$
Then,
\begin{align*}
\lim_{k\to\infty}\langle z_k,F(y_k,z_k)\rangle_{L^p(\Omega)}&=\lim_{k\to\infty}\langle |y_k|^{p-2}y_k,z_k\rangle_{L^p(\Omega)}
+\lim_{k\to\infty}\langle |z_k|^{p-2}z_k,z_k\rangle_{L^p(\Omega)}\\
&=\langle |y|^{p-2}y,z\rangle_{L^p(\Omega)}+\lim_{k\to\infty}\|z_k\|^p_{L^p(\Omega)}\\
&=\langle |y|^{p-2}y,z\rangle_{L^p(\Omega)}+\|z\|^p_{L^p(\Omega)}=\langle F(y,z),z\rangle_{L^p(\Omega)}.
\end{align*}
However, this relation implies the norm convergence $\|z_k\|_{L^p(\Omega)}\rightarrow \|z\|_{L^p(\Omega)}$. Since $z_k\to z$ weakly in $L^p(\Omega)$, we finally conclude: the sequence $\{z_k\}_{k\in\mathbb{N}}$ is strongly convergent to $z$ in $L^p(\Omega)$. By analogy, using also the lower semi-continuity  of the norm in $L^p(\Omega)$, we can verify property $(\mathfrak{A})$ just as easy.
\end{remark}
\subsection{Proof of Corollary \ref{Cor 3.18}.}
\begin{proof}
As follows from Propositions \ref{Prop 3.13} and \ref{Prop
3.20.1}, the sequence of admissible triplets
$\left\{(\mathcal{U}^\ast,y_\e,z_\e)\in
\Xi_{\e}\right\}_{\e>0}$ is relatively $\tau$-compact, and there exists a $\tau$-limit triplet
$(\mathcal{U}^\ast,y^\ast,z^\ast)$ such that $\left.
y^\ast\right|_\Omega=y_{\,\Omega,\,\mathcal{U}^\ast}$ and $\left.
z^\ast\right|_\Omega\in\mathcal{H}(y_{\,\Omega,\,\mathcal{U}^\ast})$. Having set
$y=y_{\,\Omega,\,\mathcal{U}^\ast}$, we
prove the strong convergence of $\widetilde{y}_\e$ to
$\widetilde{y}$ in $W^{1,\,p}_0(D)$. Then the strong convergence
of $z_\e$ to $z^\ast$ in $L^p(D)$ will be ensured by Corollary \ref{Rem 1.8}.

To begin with, we prove the convergence of
norms of $\widetilde{y}_\e$
\begin{equation}
\label{3.19} \|\widetilde{y}_\e\|_{W^{1,\,p}(D)}\rightarrow
\|\widetilde{y}\|_{W^{1,\,p}(D)}\ \text{ as }\ \e\to 0.
\end{equation}

As we already mentioned, since $\mathcal{U}^\ast\in U_{ad}$,  we can consider as an equivalent norm in $W^{1,\,p}_0(D)$
the following one
\[
\|y\|^{\mathcal{U}^\ast}_{W^{1,\,p}_0(D)}=\left(\int_D \left(\mathcal{U}^\ast[(\nabla
y)^{p-2}]\nabla y,\nabla y\right)_{\mathbb{R}^N}\,dx +\int_D
|y|^p\,dx\right)^{1/p}.
\]
As a result, the space
$\left<W^{1,\,p}_0(D),\|\cdot\|^{\mathcal{U}^\ast}_{W^{1,\,p}_0(D)}\right>$
endowed with this norm is uniformly convex.
Hence, instead of \eqref{3.19}, we can establish that
\begin{equation}
\label{3.20}
\|\widetilde{y}_\e\|^{\mathcal{U}^\ast}_{W^{1,\,p}(D)}\rightarrow
\|\widetilde{y}\|^{\mathcal{U}^\ast}_{W^{1,\,p}(D)}\ \text{ as }\
\e\to 0.
\end{equation}

Using the equations \eqref{2.7a} and \eqref{3.2}, we take as test
functions $\widetilde{y}$ and $\widetilde{y}_\e$, respectively.
Then, passing to the limit in \eqref{3.2}, we get
\begin{align*}
\lim_{\e\to 0}&\left(\int_D \left(\mathcal{U}^\ast[(\nabla \widetilde{y}_\e)^{p-2}]
\nabla\widetilde{y}_\e,\nabla\widetilde{y}_\e\right)_{\mathbb{R}^N}\,dx
+\int_D
|\widetilde{y}_\e|^p \,dx\right)\\
&= \lim_{\e\to
0}\left(\|\widetilde{y}_\e\|^{\mathcal{U}^\ast}_{W^{1,\,p}(D)}\right)^p=
\lim_{\e\to 0} \langle f ,\widetilde{y}_\e\rangle_{W_0^{1,p}(D)} =\langle f ,\widetilde{y}\rangle_{W_0^{1,p}(D)}
 \\
 &= \int_D \left(\mathcal{U}^\ast[(\nabla\widetilde{y})^{p-2}]
\nabla\widetilde{y},\nabla\widetilde{y}\right)_{\mathbb{R}^N}\,dx +\int_D
|\widetilde{y}|^p \,dx =
\left(\|\widetilde{y}\|^{\mathcal{U}^\ast}_{W^{1,\,p}(D)}\right)^p.
\end{align*}
Since \eqref{3.20} together with the weak convergence in
$W^{1,\,p}_0(D)$ imply the strong convergence, we arrive at the
required conclusion.
\end{proof}

\medskip
Received April 2014; revised October 2014.
\medskip
\end{document}